\documentclass[12pt]{amsart}
\usepackage{amsmath,amsthm,amsfonts,amssymb,latexsym}
\usepackage[colorlinks,pagebackref]{hyperref}
\usepackage{enumerate}
\usepackage[shortlabels]{enumitem}
\usepackage{arydshln,xcolor,mathabx}

\theoremstyle{plain}

\newtheorem{thm}{Theorem}[section]
\newtheorem{lem}[thm]{Lemma}
\newtheorem{conj}[thm]{Conjecture}
\newtheorem{pro}[thm]{Proposition}

\theoremstyle{definition}

\newtheorem{notation}[thm]{Notation}
\newtheorem{remark}[thm]{Remark}

\newtheorem{hypothesis}[thm]{Hypothesis}

\newtheorem*{thmB}{THEOREM B}
\newtheorem*{conjA}{CONJECTURE A}
\newtheorem*{thmC}{THEOREM C}

\numberwithin{equation}{section}

\newcommand{\normal}{\trianglelefteq}

\newcommand{\Syl}{\operatorname{Syl}}

\newcommand{\Aut}{\operatorname{Aut}}
\newcommand{\Ker}{\operatorname{Ker}}

\newcommand{\Irr}{\operatorname{Irr}}

\newcommand{\FF}{\mathbb{F}}

\newcommand{\C}{\mathbf{C}}

\newcommand{\bG}{\mathbf{G}}
\newcommand{\bL}{\mathbf{L}}
\newcommand{\bN}{\mathbf{N}}

\def\irrp#1{{\rm Irr}_{p'}(#1)}

\def\C{{\mathbb C}}

\def\irr#1{{\rm Irr}(#1)}
\def\irrpr#1{{\rm Irr}_{p',{\rm rel}}(#1)}

\def\cent#1#2{{\bf C}_{#1}(#2)}
\def\syl#1#2{{\rm Syl}_#1(#2)}
\def\nor{\trianglelefteq\,}

\def\zent#1{{\bf Z}(#1)}

\def\ker#1{{\rm ker}(#1)}
\def\norm#1#2{{\bf N}_{#1}(#2)}

\def\sbs{\subseteq}

\newcommand{\Out}{{\mathrm {Out}}}

\newcommand{\diag}{{\mathrm {diag}}}

\newcommand{\Spec}{{\mathrm {Spec}}}

\newcommand{\QQ}{{\mathbb Q}}
\newcommand{\ZZ}{{\mathbb Z}}

\newcommand{\bT}{\mathbf{T}}

\newcommand{\ta}{\hspace{0.5mm}^{2}\hspace*{-0.2mm}}
\newcommand{\tb}{\hspace{0.5mm}^{3}\hspace*{-0.2mm}}

\newcommand{\bC}{{\mathbf{C}}}
\newcommand{\bS}{{\mathbf{S}}}

\newcommand{\bB}{{\mathbf{B}}}
\newcommand{\bU}{{\mathbf{U}}}
\newcommand{\bZ}{{\mathbf{Z}}}
\newcommand{\bX}{{\mathbf{X}}}
\newcommand{\Al}{\textup{\textsf{A}}}
\newcommand{\Sy}{\textup{\textsf{S}}}

\begin{document}
\title{Sum of the squares of the $p'$-character degrees}

\author[Nguyen N. Hung]{Nguyen N. Hung}
\address[Nguyen N. Hung]{Department of Mathematics, The University of Akron, Akron,
OH 44325, USA}
\email{hungnguyen@uakron.edu}

\author[J. Miquel Mart\'inez]{J. Miquel Mart\'inez}
\address[J. Miquel Mart\'inez]{Departament de Matem\`atiques, Universitat de Val\`encia, 46100 Burjassot,
Val\`encia, Spain}
\email{josep.m.martinez@uv.es}

\author[Gabriel Navarro]{Gabriel Navarro}
\address[Gabriel Navarro]{Departament de Matem\`atiques, Universitat de Val\`encia, 46100 Burjassot,
Val\`encia, Spain}
\email{gabriel@uv.es}

\thanks{The first author gratefully acknowledges the support of the AMS-Simons Research Enhancement
Grant (AWD-000167 AMS). The work of the second and third authors is
supported by Grant PID2022-137612NB-I00 funded by MCIN/AEI/
10.13039/501100011033 and ERDF ``A way of making Europe." We thank
Gunter Malle for helpful conversations and for noticing gaps in the
proofs of Proposition~\ref{pro:giannelli implies A} and
Lemma~\ref{lem:SU} in an earlier version. We thank the referees for their thorough reading and their many comments and suggestions that helped improve the exposition of this paper.}



\begin{abstract}
We study the sum of the squares of the irreducible character degrees
not divisible by some prime $p$ of a finite group, and its
relationship with the corresponding quantity in a $p$-Sylow
normalizer. This leads to study a recent conjecture by E.~Giannelli,
which we prove for $p=2$ and in some other cases.
\end{abstract}

\maketitle

\section{Introduction}
There are some strikingly simple and seemingly innocent statements
that follow from the recently proven McKay conjecture \cite{CS} (or
from the techniques developed to prove it) whose validity appears to
resist  any elementary justification. For instance, W. Feit was
interested in proving that if $P$ is an abelian Sylow $p$-subgroup
of $G$, then $k(G)\ge k(\norm GP)$, where $k(G)$ is the number of
conjugacy classes of the finite group $G$ (\cite{F}). Although this
inequality now follows as a consequence of \cite{CS},  no
alternative proof is currently known. (It is worth noting that
equality happens if and only if $P\nor G$ by the It\^o--Michler
theorem.)

This current project -- long awaited by the third author -- was
expected to provide another example of such phenomena. And yet,
 it does not. What  was anticipated to be a theorem
remains, for now, a conjecture. As usual, $\irrp G$ is the set of
irreducible complex characters of the finite group $G$ whose degree
is not divisible by the prime $p$, and $P'=[P,P]$ is the derived
subgroup of the group $P$.
   \begin{conjA}\label{conj:A}
   {\sl Let $G$ be a finite group and $P \in \syl pG$.
   Then
   $$\sum_{\chi \in \irrp G} \chi(1)^2 \ge |\norm GP:P'|$$
   with equality if and only  $\norm GP$ has a normal complement in $G$.}
     \end{conjA}

The McKay theorem establishes that there exists a bijection $f:
\irrp G \rightarrow \irrp{\norm GP}$. One way to prove Conjecture A,
perhaps the most obvious,
 is to show that we can choose
$f$  to satisfy the additional condition   that $f(\chi)(1)\le
\chi(1)$ for all $\chi \in \irrp G$. This turns out to be a
conjecture by E.~Giannelli (\cite{G}), and it does not appear to
follow easily from the proof of the McKay conjecture. In
Theorem~\ref{thm:GMC reduction} below we carry out the standard
reduction of the McKay conjecture now incorporating Giannelli's
strengthening, which is therefore reduced now to a question on
simple groups.  This question is verified in a number of important
cases in Sections~\ref{sec:quasisimple} and \ref{sec:p=2} below. In
Theorem~\ref{thm:p=2Giannelli}, we completely prove it for $p=2$,
and this establishes:

       \begin{thmB}
   {\sl Conjecture A holds for $p=2$.}
     \end{thmB}

Can Conjecture A be proven independently of Giannelli's
strengthening of McKay? We do not know the answer to this question.
Characters of $p'$-degree in a group tend to have much larger degree
than those in the normalizer of a Sylow $p$-subgroup,  but so far,
no effective strategy has been developed to exploit this
observation.
\medskip

%There is a part of this work that is independent of both the McKay
%conjecture and the classification of finite simple groups, which we
%are able to resolve completely, and that takes care of the equality in Theorem A
%(assuming Giannelli's conjecture).
%In fact, we prove a relative version of it with respect to a normal subgroup.   We find surprising that a character
%restriction type of theorem admits a  relative version, since these versions regarding group
%structure  occur very rarely.

 Theorem C below, whose proof is rather involved -- but does not use the Classification of Finite Simple Groups -- takes care of the equality in Conjecture A (assuming Giannelli's conjecture), and
 might have independent interest.
%Recall that if $\theta \in \irr N$ and $N\nor G$, then
%$\irr{G|\theta}$ is the set of irreducible characters of $G$ whose
%restriction to $N$ contain $\theta$ as an irreducible constituent.
%

\begin{thmC}
{\sl Let $G$ be a finite group, $p$ a prime, and $P\in \syl pG$.
Then all irreducible characters of $p'$-degree of $\norm GP$ extend
to $G$ if and only if there is $K\nor G$ such that $K\norm GP=G$ and
$\norm KP=1$.}
%{\sl Suppose that $N \nor G$, $\theta \in \irr N$ is $G$-invariant and
%extends to $NP$, where $P \in \syl pG$. Then all the irreducible
%characters $\gamma \in \irr{N\norm GP|\theta}$ such that $p$ does
%not divide $\gamma(1)/\theta(1)$ extend to $G$ if and only if there
%is a  subgroup $N \le K \nor G$ having an irreducible character $\tau \in
%\irr K$ that extends $\theta$ such that $G=K\norm GP$ and $K\cap  N\norm GP=N$.}
    \end{thmC}

This extends \cite[Theorem B]{IsCR}, where it is required that {\sl
all} the irreducible characters of $\norm GP$ extend to $G$. As we
shall explain, in order to prove Theorem C, we shall need to
establish a relative version of it with respect to a normal
subgroup. We find it surprising that a character restriction type of
theorem admits a  relative version, since these versions regarding
group structure  occur very rarely.

\smallskip
There are some variations of Conjecture A which we do not attempt
here. Among others, it seems reasonable to replace $P'$ by the
Frattini subgroup $\Phi(P)$ and the set of $p'$-degree irreducible
characters by the  so-called {\sl almost $p$-rational} characters of
$p'$-degree. Proving this, however, seems much more complicated.

\smallskip
Concerning Giannelli's conjecture, let us mention here that in
general it is not always possible to find McKay bijections $f: \irrp
G \rightarrow \irrp{\norm GP}$ such that $f(\chi)$ is an irreducible
constituent of $\chi_{\norm GP}$, as shown by ${\sf A}_5$ for $p=2$,
or ${\rm GL}_2(3)$ for $p=3$.

\smallskip

A final but important remark is in order: Why do we care about the
sum of the squares of the irreducible character degrees not
divisible by $p$? First of all, this number would be the dimension
of any complex algebra affording the $p'$-degree irreducible
characters. Also, of the famous list of problems by Richard Brauer
(\cite{B}), Problem~2 asks when non-isomorphic finite  groups $G$
and $H$ have isomorphic group algebras $\C G$ and $\C H$. If $G$ has
a normal $p$-complement, Isaacs proved that $H$ has a normal
$p$-complement too (\cite[Theorem 7.8]{N}). The question of whether
the normality of a Sylow $p$-subgroup $P$ of $G$ is recognized by
$\C G$ remains open. In \cite{N03}, it is suggested that, perhaps,
this happens if and only if
$$\left(\sum_{\chi \in \irrp G} \chi(1)^2\right)_{p'}=|G|_{p'}\, , $$
where $n_{p'}=n/n_p$ with $n_p$ denoting the largest power of the
prime $p$ dividing the integer $n$.

\section{Proof of Theorem C}
Our notation for ordinary characters follows \cite{Is} and \cite{N}.
In this section we prove Theorem C. In order to do that, we need to
prove a relative version of it, that allows us to use induction.
This makes the proof more complicated. We start with some
preliminary results.

\begin{lem}\label{ker}
   Let $N\nor G$ and let $\theta \in \irr N$ be $G$-invariant.
   Then
   $$\ker\theta=\bigcap_{\chi \in \irr{G|\theta}} \ker\chi \, .$$
\end{lem}
\begin{proof}
It is clear that $\ker\theta$ is contained in  $\ker\chi$ for every
$\chi \in \irr{G|\theta}$, since $\chi_N=e_\chi \theta$ for $\chi
\in \irr{G|\theta}$. Suppose that $x \in \bigcap_{\chi \in
\irr{G|\theta}} \ker\chi$. Then
$\theta^G(x)=\theta^G(1)=|G:N|\theta(1) \ne 0$, and therefore $x \in N$.
Since $\theta$ is $G$-invariant, we have that
$\theta^G(x)=|G:N|\theta(x)$, and thus $\theta(x)=\theta(1)$,
 as wanted.
\end{proof}
We need an easy lemma.

\begin{lem}\label{deg}
Suppose that $K \nor G$, $H\le G$, $N=K\cap H$ and $G=KH$. Let
$\gamma \in \irr K$ be $G$-invariant, and let $\theta \in \irr N$ be
an irreducible $H$-invariant constituent of $\gamma_N$. Suppose that
restriction defines a bijection $\irr{G|\gamma} \rightarrow
\irr{H|\theta}$. Then $\gamma_N=\theta$.
\end{lem}

\begin{proof}
We have that $$\sum_{\chi \in \irr{G|\gamma}}
(\chi(1)/\gamma(1))^2=|G:K|\, .$$ We also have that $$\sum_{\tau \in
\irr{H|\theta}} (\tau(1)/\theta(1))^2=|H:N|\, .$$ By hypothesis,
$$\sum_{\tau \in \irr{H|\theta}} (\tau(1)/\theta(1))^2=
\sum_{\chi \in \irr{G|\gamma}} (\chi(1)/\theta(1))^2=
(\gamma(1)/\theta(1))^2 \sum_{\chi \in \irr{G|\gamma}}
(\chi(1)/\gamma(1))^2 \, .$$ Since $|H:N|=|G:K|$, we deduce that
$\gamma_N=\theta$.     \end{proof}

%     \begin{lem}\label{fus}
%     Let $p$ be a prime.
%     Suppose that $N \nor H \le G$, and let $\theta \in \irr N$ be linear, and faithful, where $N$ has order not divisible by $p$.
%     Assume that every $\alpha \in \irr{H|\theta}$ extends to $G$.
%     Let $h$ be a $p$-element of $H$.
%   If $h,h^g \in H$  for some $g \in G$, then $h^g$ and $h$ are $H$-conjugate.
%     \end{lem}
%
%     \begin{proof}
%     Since $N$ is a central $p'$-subgroup of $H$ and $h$ is a $p$-element,
%  we have that $\cent{H/N}{hN}=\cent Hh/N$.  Hence,
%  by Lemma 5.13 of \cite{N}, we have that $h$ is $\theta$-good.
%  For each $\alpha \in \irr{H|\theta}$, let $\hat\alpha \in \irr G$ be an extension
%  of $\alpha$ to $G$. Then $\alpha(h^g)=\hat\alpha(h^g)=\hat\alpha(h)=\alpha(h)$.
%  Now, by Theorem 5.21 of \cite{N}, we have that
%  $$\sum_{\alpha \in \irr{H|\theta}} \alpha(h^g)\overline{\alpha(h)}=
%  \sum_{\alpha \in \irr{H|\theta}} |\alpha(h)|^2=|\cent {H/N}{hN} \hat\theta(h) \, ,$$
%  where $\hat\theta(h) \ne 0$ (by the definition of $\hat\theta$ before Theorem 5.21 of \cite{N}). We conclude that $hN$ and $h^gN$
%  are $H/N$. Then, using that $N$ is a central $p'$-subgroup,
%  we have that $h$ and $h^g$ are $H$-conjugate
%
%     \end{proof}

\begin{lem}\label{r}
Suppose that $G=KP$, where $K$ is a normal $p$-complement and $P\in
\syl pG$. Let $\gamma \in \irr G$. If $P\sbs \ker\gamma$, then
$\gamma_{\cent KP}$ is irreducible.
\end{lem}

\begin{proof}
Since $\ker\gamma \nor G$, we have that $[K,P]P \sbs \ker\gamma$.
Thus $\gamma_K$ is an irreducible character of $K$ with $[K,P]$
contained in its kernel. By coprime action (for instance, Lemma 4.28 of \cite{Is08}),  we have that $\cent KP[K,P]=K$ and
thus $\gamma_{\cent KP}$ is irreducible.
\end{proof}

\begin{lem}\label{unique}
Suppose that $N \nor G$, $P \in \syl pG$. Let $N\le K_i\nor G$ be
complementing $N\norm GP/N$ for $i=1,2$ in $G/N$. Then $K_1=K_2$.
\end{lem}

\begin{proof}
By working in $G/N$, we may assume that $N=1$.
We have that $G/K_i$ has a normal Sylow $p$-subgroup.
 Let $L=K_1 \cap K_2$. Then $G/L$ has a normal
Sylow $p$-subgroup.  Hence, $L\norm GP=G$. Then
$K_i=K_i \cap L\norm GP=L(K_i \cap \norm GP)=L$, as wanted.
\end{proof}

In several occasions, we shall use Tate's theorem in the following
form.
\begin{thm}\label{tate}
Suppose that $G$ is a finite group, $P \in \syl pG$, and $K\nor G$.
If $K\cap P=K\cap P'$, then $K$ has a normal $p$-complement.
\end{thm}

\begin{proof}
We may assume that $G=KP$. In the notation of \cite[Theorem
6.31]{Is}, we have that $P\cap {\bf A}^p(G)={\bf A}^p(P)=P'$, where
${\bf A}^p(G)$ is the smallest normal subgroup $L$ of $G$ such that
$G/L$ is an abelian $p$-group. It then follows that $G$ has a normal
$p$-complement by \cite[Theorem 6.31]{Is}. Therefore, $K$ has a
normal $p$-complement.
\end{proof}

Let $N\nor G$ and let $\theta \in \irr N$ be $G$-invariant. We let
\[\irrpr{G|\theta}:=\{ \chi \in \irr{G|\theta} \mid p \text{ does not
divide } \chi(1)/\theta(1)\}.\] The following easily implies
Theorem~C (when $N$ is trivial).

\begin{thm}\label{thmB}
Let $N\nor G$ and let $\theta \in \irr N$ be $G$-invariant. Let $P
\in \syl pG$. Assume that $\theta$ extends to $NP$. Then the
following are equivalent.
\begin{enumerate}
   \item
   Every $\chi \in {\rm Irr}_{p',{\rm rel}}(N\norm GP|\theta)$
   extends to $G$.
   \item
   Every $\chi \in {\rm Irr}(N\norm GP|\theta)$
   extends to $G$.

   \item
   There is a normal complement $K/N$
  to $\norm GP N/N$ in $G/N$ such that $\theta$ has an extension
 $\hat\theta \in \irr K$.
  \item
There is a normal complement $K/N$ to $\norm GP N/N$ in $G/N$ such
that $\theta$ has a $G$-invariant extension $\hat\theta \in \irr K$.
  \end{enumerate}
   \end{thm}

\begin{proof}
First we prove that (iv) and (iii) are equivalent. We only have to
prove that (iii) implies (iv). Since $PN \cap K=N$, we have that
$K/N$ is a group of order not divisible by $p$. By Sylow theory,
$\norm G{PN}=\norm GP N$, and therefore
$\cent{K/N}P=\cent{K/N}{PN/N} \sbs K/N \cap \norm{G/N}{PN/N}=1$. Let
$\eta \in \irr K$ be an extension of $\theta$, and let $\Delta$ be
the set of extensions of $\theta$ to $K$. We have that
$\Delta=\{\lambda \eta \, |\, \lambda \in \irr{K/N}$ is linear$\},$
by Gallagher's theorem (\cite[Corollary 6.17]{Is}). Therefore
$|\Delta|$ is not divisible by $p$. Since $P$ acts on $\Delta$ by
conjugation, because $\theta$ is $G$-invariant, by counting we have
that there is some $P$-invariant $\hat\theta \in \Delta$ extending
$\theta$. Since $\cent{K/N}P=1$, we have that $\hat\theta$ is the
unique $P$-invariant character of $K$ lying over $\theta$
 (using \cite[Theorem 13.31 and Problem 13.10]{Is}. We shall use this argument
 several times).
 We
claim that $\hat\theta$ is also $G$-invariant. It is enough to show
that $\hat\theta$ is $\norm GP$-invariant. If $n \in \norm GP$, then
$\hat\theta^n$ is a $P$-invariant extension of $\theta^n=\theta$. By
uniqueness, we have that $\hat\theta^n=\hat\theta$, as wanted.

To prove that (iv) implies (ii), we apply \cite[Lemma 6.8(d)]{N}.

Since it is clear that (ii) implies (i), to complete the proof of
the theorem,
 it is enough to show that
(i) implies (iii). We argue this by induction on $|G:N|$. Recall
that $N\norm GP/N=\norm{G/N}{PN/N}$. By using the theory of
character triple isomorphisms (see \cite[Chapter 11]{Is}), we may
assume that $N\sbs \zent G$. Hence, $N\sbs \norm GP$.
In particular, $\theta$ is linear.
Write $\theta=\theta_p\theta_{p'}$, where $\theta_p$ has order a
power of $p$ and $\theta_{p'}$ has order not divisible by $p$.
 Since $\theta$ extends to $PN$ by hypothesis,   therefore
so does $\theta_p$, which is a power of $\theta$. Hence, $\theta_p$
extends to some linear character $\nu \in \irr G$ (by \cite[Theorem
6.26]{Is}).
If $\chi  \in {\rm Irr}_{p',{\rm rel}}(\norm GP|\theta_{p'})$,
then $\chi\nu_{\norm GP}
\in {\rm Irr}_{p',{\rm rel}}(\norm GP|\theta)$.
If $\eta \in \irr G$ extends $\chi\nu_{\norm GP}$, then $\eta \nu^{-1}$
extends $\chi$.  Therefore, we may assume that
$\theta$ is a linear character of $p'$-order. By modding out by
$\ker\theta$, we may assume that $N$ is a $p'$-group and that
$\theta$ is faithful. Hence, our hypothesis is that every $\chi \in
\irrp{\norm GP|\theta}$ extends to $G$.  We want to show that there
is $K \nor G$ complementing $\norm GP/N$ in $G$, and that  $\theta$
extends to $K$.  For each $\tau \in \irrp{\norm GP| \theta}$, fix
$\tilde\tau \in \irr{G}$ an extension of $\tau$ to $G$.

If $P' \sbs X \le P$ and $X \nor \norm GP$, we claim that there
exists $L\nor G$ such that $L\cap \norm GP=NX$. Let
$\tilde\theta=\theta \times 1_{X}$. Notice that all $\irr{\norm
GP|\tilde\theta}$ have $p'$-degree because $\norm GP/X$ has
an abelian normal Sylow $p$-subgroup $P/X$. Let
$$U=\bigcap_{\tau \in \irr{\norm GP|\tilde\theta}} \ker{\tilde\tau} \nor G \, .$$
Then, using Lemma \ref{ker}, we have that $U\cap \norm
GP=\ker{\tilde\theta}=X$. Let $L=UN \nor G$. Therefore $L\cap \norm
GP=NX$, as claimed. Notice that in this case, $X=P\cap L \in \syl p
L$.

\smallskip

By letting $X=P'$ in the claim in the previous paragraph of this
proof, let
 $L\nor G$ such that $L\cap \norm GP=NP'$. Since $P\cap L=P'$, by Tate's theorem \ref{tate},
  we have that $L$ has  a normal $p$-complement $K$.
  Also,   $P'$ is a Sylow $p$-subgroup of $L$. Let $W=K\norm GP=L\norm GP$.
 Notice that $\cent{K/N}P=1$ since $\norm KP=N$, and therefore
 over $\theta$ there is a unique
 $P$-invariant $\eta \in \irr K$ (again using \cite[Theorem 13.31 and Problem 13.10]{Is}). By uniqueness, notice that $\eta$ is $\norm GP$-invariant.  We have that
 $$|\irrp{W|\theta}|=|\irrp{\norm GP|\theta}|$$
by the relative McKay conjecture (for $p$-solvable groups),
\cite[Theorem 10.26]{N}. (Notice that $W$ is indeed $p$-solvable, so
we have not used the solution of the general McKay conjecture.) We claim that $\irrp{W|\theta}=\{\tilde\tau_W
\, |\, \tau \in \irrp{\norm GP|\theta}\}$.
Indeed, let $f:\irrp{\norm GP|\theta} \rightarrow \irrp{W|\theta}$
given by $f(\tau)=\tilde\tau_W$. Since $(\tilde\tau)_{\norm GP}=\tau$,
we do have that $f(\tau) \in \irrp{W|\theta}$. Since $f$ is clearly injective,
it is necessarily bijective and the claim follows.
%   (Also, it needed, we have that
%   $\irrp{G|\theta}=\{\tilde\tau  \, |\, \tau \in \irrp{\norm GP|\theta}\}$ by the McKay conjecture.)
Since $\eta$ is the only $P$-invariant irreducible
character of $K$ over $\theta$, it easily
follows that $\irrp{G|\theta}=\irrp{G|\eta}$ and
$\irrp{W|\theta}=\irrp{W|\eta}$.  Indeed,
if $\chi \in \irrp{G|\theta}$, then $\chi_K$ has some
$P$-invariant irreducible constituent, using that $\chi$ has $p'$-degree.
All irreducible constituents of $\chi_K$ lie over $\theta$,
so we deduce that this $P$-invariant constituent should be
$\eta$. Notice then that every ${\rm
Irr}_{p'}(W|\eta)$ extends to $G$.
We claim now that $\eta$ is $G$-invariant.
First notice that since $\eta$ is $P$-invariant and $K$ is
a $p'$-group, we have that $\eta$ extends to
$KP$ (using Corollary 6.2 of \cite{N}). Hence, there is
some $p'$-degree irreducible character of $W$ over $\eta$.
Therefore, there exists $\tau \in \irrp{\norm GP|\theta}$ such that
$\tilde\tau_W$ lies over $\eta$. Since $\eta$ is invariant in $W$,
it follows that $W \le G_\eta$. If $\epsilon $ is the Clifford correspondent of
$\tilde\tau$ over $\eta$, it follows that $\epsilon^G=\tilde\tau$. Since
$\tilde\tau_{G_\eta}$ is irreducible, necessarily $G_\eta=G$.
Hence, if $K>N$, we can apply induction to $\irrp{W|\eta}$ with respect to $G$.

Suppose first  that $P' \nor G$, and  assume next that $P'>1$.
Working in $\bar G=G/P'$ and using induction, we conclude that there
is $P' \le R\nor G$ such that $R\norm GP=G$, and $R\cap \norm
GP=NP'$, and that $\theta \times 1_{P'}$ extends to $\gamma \in
\irr{R}$. Since $P\cap R=P'$, we have that $R$ has a normal
$p$-complement $S$ by Tate's Theorem \ref{tate}, and Sylow
$p$-subgroup $P'$. This normal $p$-complement complements $\norm
GP/N$ in $G/N$.   Since $\gamma_N=\theta$, we have that $\gamma_S$ extends
$\theta$ and we are done, in the case that $P'\nor G$ and $P'>1$.
Assume now that $P$ is abelian.
Hence, we have that all $\irr{\norm GP|\theta}=\irrp{\norm GP|\theta}$ extend to $G$, by hypothesis. By the claim in the fourth paragraph
of this proof (letting $X=P$), let $Y \nor G$ such that $Y \cap \norm GP=NP$. Then
$P \in  \syl pY$ and $G=Y \norm GP$ by the Frattini argument. Now,
$P \sbs \zent {\norm YP}$, and by Burnside's $p$-complement theorem
(Theorem 5.13 of \cite{Is08}),
we have that $Y$ has a normal $p$-complement $Q$. Then $Q\norm GP=G$
and $Q\cap \norm GP=N$. Notice that $G$ is $p$-solvable. Since
$\cent{Q/N}P=1$, let $\mu \in \irr Q$ be the unique $P$-invariant
over $\theta$. By uniqueness,  $\mu$ is $\norm GP$-invariant, and
therefore $G$-invariant. Also
$\irrp{G|\theta}=\irrp{G|\mu}=\irr{G|\mu}$, using that $\mu$ extends
to $QP=Y$, Gallagher \cite[Corollary 6.17]{Is}, and the fact that
$P$ is abelian. By the relative version of the McKay conjecture (in
$p$-solvable groups) \cite[Theorem 10.26]{N}, we have that
$$|\Irr(G|\mu)|=|\irrp{G|\mu}|=|\irrp{G|\theta}|=|\irrp{\norm GP|\theta}|=
|\Irr(\norm G
P|\theta)| \, .$$
By a previous argument, we see that
${\rm Irr}(G|\mu)=\{\tilde\tau \mid \tau \in \irr{\norm GP|\theta}\}$, and we see that restriction defines a bijection
$\irr{G|\mu} \rightarrow \irr{\norm GP|\theta}$.  By Lemma
\ref{deg}, we have that $\mu_N=\theta$, and we are done. Therefore
we may assume that $P'$ is not normal in $G$.

Hence, using the notation of the fifth paragraph of this proof, we
may assume that $K>N$. (Otherwise, $L=NP'=N \times P'$ and
necessarily $P' \nor G$.) We have that every ${\rm
Irr}_{p'}(W|\eta)$ extends to $G$. By induction, there exists $K\le
E \nor G$ complementing $W/K$ and some $\rho \in \irr{E}$ such that
$\rho_K=\eta$. Notice that $E$ complements $\norm GP/N$. In
particular $\cent{E/N}P=1$, and $\rho$ is the only $P$-invariant
over $\theta$. Hence ${\rm Irr}_{p'}(G|\theta)={\rm
Irr}_{p'}(G|\rho)$. By a previous argument, we know that $\rho$ is
$\norm GP$-invariant, and therefore $G$-invariant. We only need to
show that $\rho_N=\theta$. Recall that $G$ is $p$-solvable. Hence
$|\irrp{G|\rho}|=|\irrp{G|\theta}|=|\irrp{\norm GP|\theta}|$
(again by \cite[Theorem 10.26]{N}), and
therefore we deduce that
$$\irrp{G|\rho}=\{\tilde\tau \, |\, \tau \in \irr{\norm GP|\theta} \}\, .$$

Using that $E$ is a $p'$-group, let $\hat\rho \in \irr{EP}$ be an
extension of $\rho$. Since $\hat\rho$ has $p'$-degree and $EP \nor
G$ (because $E \norm GP=G$), it follows that $\hat\rho$ lies under
some $p'$-irreducible character of $G$, call it $\chi$. Hence
$\hat\rho$ lies under some $\tilde\tau=\chi$ for some $\tau \in
\irrp{\norm GP|\theta}$. However $\tau_{P'}$ is a multiple of
$1_{P'}$. Thus $\hat\rho_{P'}$ is a multiple of $1_{P'}$. Write $C=\cent E{P'}$. By Lemma
\ref{r} applied to $EP'$, we have that $\varphi=\hat\rho_{C}=\rho_C$ is irreducible.
 Hence, we have that restriction defines a
bijection $\irr{G|\rho} \rightarrow \irr{\norm G{P'}|\varphi}$, by
\cite[Lemma 6.8(d)]{N}. Notice that $\irr{\norm
G{P'}|\varphi}=\irr{\norm G{P'}|\theta}$ since $\cent{C/N}P=1$. By
the $p$-solvable case of the McKay conjecture \cite[Theorem
10.26]{N} we know that $|\irrp{\norm G{P'}|\theta}|=|\irrp{\norm
GP|\theta}|$. Therefore
$$\irrp{\norm G{P'}|\theta}=\{ \tilde\tau_{\norm G{P'}} |\tau \in \irrp{\norm GP|\theta}\} \, .$$
Since $\norm G{P'}<G$, by induction, we have that $\norm GP/N$ has a normal complement in
$\norm G{P'}/N$ which by Lemma \ref{unique} has to be $C$ and that
$\theta$ extends to $C$. By the first paragraph of this proof, we
have that $\theta$ has a $P$-invariant extension to $C$.
Since $\cent{C/N}P=1$, then this
$P$-invariant extension should be $\varphi=\rho_C$, and therefore
$\rho$ extends $\theta$, as desired.
\end{proof}

\section{The McKay conjecture and inequality between character degrees}

The following refinement of the McKay conjecture has been proposed
in \cite{G}.

\begin{conj}\label{conj:Giannelli-McKay}
Let $G$ be a finite group, $p$ a prime and $P\in\Syl_p(G)$. Then
there is a bijection
$$^*:\Irr_{p'}(G)\rightarrow\Irr_{p'}(\norm G P)$$
with \[\chi^*(1)\leq\chi(1).\] for all $\chi\in\irrp{G}$.
\end{conj}

We observe now that if Conjecture~\ref{conj:Giannelli-McKay} is
true, then so is Conjecture A.
%(it is also true that its
%relative version follows from the relative version of
%Conjecture~\ref{conj:Giannelli-McKay}, see Theorem~\ref{thmA}).
\begin{pro}\label{pro:giannelli implies A}
Let $G$ be a finite group, $p$ a prime and $P\in\Syl_p(G)$. If
Conjecture~\ref{conj:Giannelli-McKay} is true for $G$ then
$$\sum_{\chi \in \irrp G} \chi(1)^2 \ge |\norm GP:P'|$$
   with equality if and only  $\norm GP$ has a normal complement in $G$.
\end{pro}
\begin{proof}
Let $\Omega:\irrp{G}\rightarrow\irrp{\norm G P}$ be the bijection
from Conjecture \ref{conj:Giannelli-McKay}. Then
$$\sum_{\chi \in \irrp G} \chi(1)^2 \ge \sum_{\chi\in\irrp G}
\Omega(\chi)(1)^2=\sum_{\psi\in\irrp{\norm G P}}\psi(1)^2 =|\norm
GP:P'|$$ and the inequality part follows. If $$\sum_{\chi \in \irrp
G} \chi(1)^2 =|\norm GP:P'|$$ then for every $\chi\in\irrp{G}$ we
have $\Omega(\chi)(1)=\chi(1)$. We claim  that this implies that
restriction is a bijection $\irrp{G}\rightarrow\irrp{\norm G P}$.
Write $H=\norm GP$. Suppose that $a_1=1< \ldots < a_k$ are the
degrees of $\irrp G$, and therefore, of $\irrp H$, occurring with
multiplicities $m_1, \ldots, m_k$, respectively. Let $\{\psi_1,
\ldots, \psi_{t}\}$ and $\{\chi_1, \ldots, \chi_t\}$
 be the irreducible characters in $\irrp{H}$ and $\irrp G$ of maximal degree $a_k$,
 where $t=m_k$.
 Now, $(\psi_i)^G$ contains an irreducible character of $p'$-degree of $G$ of degree at least $\psi_i(1)$,
using that $(\psi_i)^G(1)$ is not divisible by $p$.
 Necessarily, $\psi_i$ contains some $\chi_j$. Then $(\chi_j)_{H}=\psi_i$, and we can reorder
 so that $(\chi_i)_{H}=\psi_i$.
 Suppose that $\{\delta_1, \ldots, \delta_s\}$ and $\{\eta_1, \ldots, \eta_s\}$
 are the irreducible characters in $\irrp{H}$ and $\irrp G$ of the next degree $a_{k-1}$, with $s=m_{k-1}$.
 Again $(\delta_i)^G$ contains a $p'$-irreducible character $\tau$ with degree at least $a_{k-1}$.
 If $\tau(1)=a_k$, then we know that $\tau_{H}$ is irreducible and then $\tau_H=\delta_i$. This is impossible.
 Hence, we have that $(\delta_i)^G$ contains some $\eta_j$, which necessarily  extends $\delta_i$.
 Reordering, we may assume that $(\eta_i)_H=\delta_i$.  We proceed like this until the claim is proven.

We may now apply Theorem \ref{thmB} with $N=1$ to conclude.
\end{proof}

The purpose of this section is to give a reduction of Conjecture
\ref{conj:Giannelli-McKay} to simple groups. The following is a
slight refinement of the inductive McKay condition, and will be the
condition we impose on quasisimple groups in our reduction. We use
the \emph{central isomorphism} relation $\geq_c$ from
\cite[Definition 10.14]{N}, which generalizes the notion of a
character triple isomorphism.

\begin{conj}\label{conj:inductive Giannelli-McKay}
Let $S$ be a quasisimple group with cyclic center, $P\in\Syl_p(S)$
and $A=\Aut(S)_P$. Then there is an $A$-stable subgroup $\norm S
P\leq M < S$ and an $A$-equivariant bijection
$$\Psi:\Irr_{p'}(S)\rightarrow\Irr_{p'}(M)$$ with
\[\Psi(\chi)(1)\leq\chi(1)\] and $$(S\rtimes
A_\chi,S,\chi)\geq_c(M\rtimes A_\chi, M, \Psi(\chi))$$ for every
$\chi\in\Irr_{p'}(S)$.
\end{conj}

\begin{thm}\label{thm:perfect}
Let $K\normal G$ be perfect, and assume $K/\zent K\cong X^n$ where
$X$ is a finite simple group and $\zent K$ cyclic. Assume every
quasisimple group $S$ with $S/\zent S\cong X$ and $\zent S$ cyclic
satisfies Conjecture~\ref{conj:inductive Giannelli-McKay}. Let
$Q\in\Syl_p(G)$ and $R=Q\cap K$. Then there exists a $\norm G
R$-stable subgroup $\norm K R\leq H < K$ and a $\norm G
R$-equivariant bijection
$$\Psi:\Irr_{p'}(K)\rightarrow\Irr_{p'}(H)$$ such that \[\Psi(\theta)(1)\leq\theta(1)\] and
$$(G_\theta,K,\theta)\geq_c(H\norm G R_\theta, H,\Psi(\theta))$$
for all $\theta\in\Irr_{p'}(K)$.
\end{thm}

\begin{proof}
We have that $K$ is a central product $S_1\cdots S_n$ where
$\zent{S_i}=\zent K$, each $S_i$ is perfect and $S_i/\zent{S_i}\cong
X$. Use the proof of \cite[Theorem~10.25]{N} using the $S_i$'s
instead of the universal covering group of $X$, and noticing that
the construction of the bijection $\Psi$ satisfies the degree
inequality if one assumes Conjecture~\ref{conj:inductive
Giannelli-McKay}. Note that the subgroup $H$ is constructed by
taking the product of the subgroups $M$ appearing in the statement
of the Conjecture~\ref{conj:inductive Giannelli-McKay}.
\end{proof}

The next result follows the proof of \cite[Theorem 10.26]{N}. We
sketch it for the reader's convenience. Recall that we say a nonabelian finite
simple group $S$ is \emph{involved} in $G$ if there exist $H\normal K\leq G$ with $K/H\cong S$.

\begin{thm}\label{thm:GMC reduction}
Assume that, for every simple group $X$ of order divisible by $p$
involved in $G$, Conjecture~\ref{conj:inductive Giannelli-McKay}
holds whenever $S$ is a quasisimple group with $S/\zent S\cong X$
and $\zent S$ is cyclic. Let $Z\normal G$, $P\in\Syl_p(G)$, $\lambda\in\Irr(Z)$ and
assume $\lambda$ is $P$-invariant. Then there is a bijection
$$\Omega:\Irr_{p'}(G|\lambda)\rightarrow\Irr_{p'}(Z\norm G P|\lambda)$$
with $\Omega(\chi)(1)\leq\chi(1)$ for all $\chi\in\Irr_{p'}(G)$.
\end{thm}

\begin{proof}
We argue by induction on $|G:Z|$. Notice that
$|G:G_\lambda|\geq|Z\norm G P:(Z\norm G P)_\lambda|$, so that by the
Clifford correspondence we may assume $\lambda$ is $G$-invariant. By
character triple isomorphisms we may assume $Z$ is central and
cyclic and $\lambda$ is linear and faithful. Let $L/Z$ be a chief
factor of $G/Z$. Let $\Delta$ be a $G$-transversal on the set of
$P$-invariant characters in $\Irr(L|\lambda)$ lying under some $\chi\in\Irr_{p'}(G)$ and notice that
\cite[Lemma~9.3]{N} implies that $\Delta$ is also a $\norm G
P$-transversal on the $P$-invariant characters in $\Irr(L|\lambda)$ lying under some $\chi\in\Irr_{p'}(L\norm G P)$.
This implies that
$$\Irr_{p'}(G|\lambda)=\bigsqcup_{\theta\in\Delta}\Irr_{p'}(G|\theta)$$ and
$$\Irr_{p'}(L\norm G P|\lambda)=\bigsqcup_{\theta\in\Delta}\Irr_{p'}(L\norm G P|\theta).$$
By induction, for every $\mu\in\Delta$ there is a bijection
$$\Omega_\mu:\Irr_{p'}(G|\mu)\rightarrow\Irr_{p'}(L\norm G P|\mu)$$
satisfying $\Omega_\mu(\chi)(1)\leq\chi(1)$ for every
$\chi\in\Irr_{p'}(G|\mu)$. We define
$\Omega_1:\Irr_{p'}(G|\lambda)\rightarrow\Irr_{p'}(L\norm G
P|\lambda)$ by $\Omega_1(\chi)=\Omega_\mu(\chi)$ if
$\chi\in\Irr(G|\mu)$ and we have $\Omega_1$ is a bijection
satisfying $\Omega_1(\chi)(1)\leq\chi(1)$ for every
$\chi\in\Irr_{p'}(G|\lambda)$. If $L\norm G P<G$ then by induction
we have a bijection $$\Omega_0:\Irr_{p'}(L\norm G
P|\lambda)\rightarrow\Irr_{p'}(\norm G P|\lambda)$$ also satisfying
$\Omega_0(\psi)(1)\leq\psi(1)$ for every
$\psi\in\Irr_{p'}(G|\lambda)$. We take
$\Omega:=\Omega_0\circ\Omega_1$ and this bijection satisfies the
desired properties. Therefore we may assume $L\norm G P=G$. In
particular, $L/Z$ is not a $p$-group.

If $L/Z$ is a $p'$-group then we write $Z=Z_p\times Z_p'$ where
$Z_p\in\Syl_p(Z)$ and $\mu=\lambda_{Z_{p'}}$ and
$\nu=\lambda_{Z_p}$. Let $K\normal L$ be a $p$-complement of $L$. If
$\Delta$ is a $G$-transversal on the set of $P$-invariant characters
of $\Irr(K|\mu)$ and $^*:\Irr_P(K)\rightarrow\Irr(C)$ is the
$P$-Glauberman correspondence, where $C=\cent K P$, then notice that
this correspondence restricts to a bijection
$^*:\Irr_P(K|\mu)\rightarrow\Irr(C|\mu)$. We have
$$\Irr_{p'}(G|\lambda)=\bigsqcup_{\theta\in\Delta}\Irr_{p'}(G|\theta\times\nu)$$ and
$$\Irr_{p'}(\norm G P|\lambda)=\bigsqcup_{\theta\in\Delta}\Irr_{p'}(\norm G P|\theta^*\times\nu)$$
again applying \cite[Lemma 9.3]{N}. Now by \cite[Theorem~6.5]{T08}
and \cite[Theorem~7.12]{T09} we have that the character triples
$(G_\theta,K,\theta)$ and $(\norm G P_{\theta^*},C,\theta^*)$ are
isomorphic, which implies that there is a bijection
$$\Phi_\theta:\Irr(G_\theta|\theta)\rightarrow\Irr(\norm G P_{\theta^*}|\theta^*)$$
preserving character degree ratios (and in particular, restricting
to $p'$-degree characters). Since $\theta^*(1)\leq \theta(1)$, we
have $\Phi_\theta(\psi)(1)\leq\psi(1)$ for all
$\psi\in\Irr_{p'}(G_\theta|\theta)$. Further, it follows from the
definition of character triple isomorphism that $\Phi_\theta$
restricts to a bijection
$$\Phi_\theta:\Irr_{p'}(G_\theta|\theta\times\nu)\rightarrow\Irr_{p'}(\norm G P_{\theta^*}|\theta^*\times\nu)$$ with $\Phi_\theta(\psi)(1)\leq\psi(1)$ for all $\psi\in\Irr_{p'}(G_\theta|\theta\times\nu)$. Since $G_\theta\cap\norm G P=\norm G P_{\theta^*}$, using the Clifford correspondence we find a bijection
$$\hat\Phi_\theta:\Irr_{p'}(G|\theta\times\nu)\rightarrow\Irr_{p'}(\norm G P|\theta^*\times\nu)$$ with $\hat\Phi_\theta(\chi)(1)\leq\chi(1)$ for all $\chi\in\Irr_{p'}(G_\theta|\theta\times\nu)$ and we are done by defining
$$\Omega:\Irr_{p'}(G|\lambda)\rightarrow\Irr_{p'}(\norm G P|\lambda)$$
by $\Omega(\chi):=\hat\Phi_\theta(\chi)$ where $\theta\in\Delta$
lies below $\chi$.

Therefore we may assume $L/Z$ is semisimple. Let $K=L'$, $Z_1=Z\cap
K=\zent{K}$, and $\nu=\lambda_{Z_1}$. Then $K$ is perfect and
satisfies the hypotheses of Theorem~\ref{thm:perfect}. Let $R=P\cap
K$. There exists a $\norm G R$-stable subgroup $\norm K R\leq H<K$
and an $\norm G R$-equivariant bijection
$$\Psi:\Irr_{p'}(K)\rightarrow\Irr_{p'}(H)$$ satisfying
$\Psi(\eta)(1)\leq\eta(1)$ and inducing central character triple
isomorphisms. By the definition of $\geq_c$ we see that $\Psi$
restricts to a bijection
$$\Psi:\Irr_{p'}(K|\nu)\rightarrow\Irr_{p'}(H|\nu)$$ with the same properties. The character triple isomorphisms induce bijections $$\Phi_\mu:\Irr_{p'}(G_\mu|\mu)\rightarrow\Irr_{p'}(H\norm G R_{\Psi(\mu)}|\Psi(\mu))$$ which satisfy $\Phi_\mu(\chi)(1)\leq\chi(1)$, and send characters over the central product $\mu\cdot\lambda\in\Irr_{p'}(L)$ to characters over $\Psi(\mu)\cdot\lambda\in\Irr_{p'}(HZ)$. Again, $|G:G_\mu|\geq|H\norm G R:M\norm G R_{\Psi(\mu)}|$ so by the Clifford correspondence and the above remark we may find a bijection
$$\hat\Phi_\mu:\Irr_{p'}(G|\mu\cdot\lambda)\rightarrow\Irr_{p'}(H\norm G R|\Psi(\mu)\cdot\lambda)$$ satisfying
$\hat\Phi_\mu(\chi)(1)\leq\chi(1)$ for all
$\chi\in\Irr_{p'}(G|\mu\cdot\lambda)$. Now by taking transversals
over the $P$-invariant characters in $\Irr_{p'}(K|\nu)$ that lie under characters $\chi\in\Irr_{p'}(G)$ and arguing
as before we obtain a bijection
 $$\Omega_0:\Irr_{p'}(G|\lambda)\rightarrow\Irr_{p'}(H\norm G R|\lambda)$$ satisfying $\Omega_0(\chi)(1)\leq\chi(1)$ for every $\chi\in\Irr_{p'}(G|\lambda)$. Since $H<K$ and $L\norm G P=G$ we have $\norm G P\leq H\norm G R<G$ so, by induction we have a bijection
 $$\Omega_1:\Irr_{p'}(H\norm G R|\lambda)\rightarrow\Irr_{p'}(\norm G P|\lambda)$$ satisfying $\Omega_1(\psi)(1)\leq \psi(1)$ for all $\psi\in\Irr_{p'}(H\norm G R|\lambda)$ and the result follows by taking $\Omega:=\Omega_1\circ\Omega_0$.
\end{proof}

Conjecture~\ref{conj:Giannelli-McKay} can be recovered by setting
$Z=1$ in Theorem~\ref{thm:GMC reduction}. By arguing as in
Proposition~\ref{pro:giannelli implies A} and using
Theorem~\ref{thm:GMC reduction} we obtain a relative version of
Conjecture~A. If $N\nor G$ and $\theta\in\irr N$ recall that we write
$$\irrpr{G|\theta}=\{\chi\in\irr{G|\theta}\mid
p \text{ does not divide } \chi(1)/\theta(1)\}.$$
\begin{pro}
Let $G$ be a finite group and $P \in \syl pG$. Let $N \nor G$, and
let $\theta \in \irr N$ be $G$-invariant such that $\theta$ extends
to $NP$. Assume Conjecture~\ref{conj:inductive Giannelli-McKay}
holds for every covering group of every simple group of order
divisible by $p$ involved in $G$. Then
$$\sum_{\chi \in \irrpr{G|\theta}} (\chi(1)/\theta(1))^2 \ge \sum_{\tau \in \irrpr{N\norm GP|\theta}} (\tau(1)/\theta(1))^2$$
with equality if and only  there is a normal complement $K/N$ to
$\norm GP N/N$ in $G/N$ such that $\theta$ has an extension
$\hat\theta \in \irr K$.
\end{pro}

%%%%%%%%%%%%%%%%%%%%%%%%%%%%%%%%%%%%%%%
%%%%%%%%%%%%%%%%%%%%%%%%%%%%%%%%%%%%%%%
%%%%%%%%%%%%%%%%%%%%%%%%%%%%%%%%%%%%%%%
%%%%%%%%%%%%%%%%%%%%%%%%%%%%%%%%%%%%%%%

\section{Quasisimple groups}\label{sec:quasisimple}

In this section and the next, we present evidence supporting
Conjecture~\ref{conj:inductive Giannelli-McKay}. We verify the
conjecture for, among other cases, all quasisimple groups of
exceptional Lie type with respect to all primes, as well as for all
groups of Lie type defined in characteristic equal to the given
prime $p$. These results, in particular, allow us to confirm the
conjecture for all quasisimple groups when $p=2$.

In the cases mentioned, we in fact prove a slightly stronger version
of Conjecture~\ref{conj:inductive Giannelli-McKay}, stated as
follows.

\begin{conj}\label{thm:simple}
Let $S$ be a quasisimple group with cyclic center, $P\in\Syl_p(S)$
and $A=\Aut(S)_P$. Then there exists an $A$-equivariant bijection
$$\Psi:\Irr_{p'}(S)\rightarrow\Irr_{p'}(\bN_S(P))$$ such that
\[\chi(1)\geq\Psi(\chi)(1)\] and $$(S\rtimes
A_\chi,S,\chi)\geq_c(\bN_S(P)\rtimes A_\chi,\bN_S(P) , \Psi(\chi))$$
for every $\chi\in\Irr_{p'}(S)$.
\end{conj}

The existence of an $A$-equivariant bijection $\Psi$ from
$\Irr_{p'}(S)$ to $\Irr_{p'}(\bN_S(P))$ satisfying the so-called
\emph{central isomorphism} $(S\rtimes
A_\chi,S,\chi)\geq_c(\bN_S(P)\rtimes A_\chi,\bN_S(P) , \Psi(\chi))$
(of character triples) has been established in several papers and
was completed in \cite{CS} by Cabanes and Sp\"{a}th. In fact, by
\cite[Theorem~B]{CS}, we now know that such bijection, which we
shall refer to as a \emph{McKay-good} bijection, exists for all
finite groups. It is the extra \emph{degree condition}
$\Psi(\chi)(1)\leq\chi(1)$ that we need to consider in this paper.

Following the literature, we say that a quasisimple group $S$
satisfies the \emph{inductive McKay condition} if it satisfies
Conjecture~\ref{conj:inductive Giannelli-McKay}, possibly excluding
the degree condition. Note that there are other equivalent
formulations of the inductive McKay conditions; see, for example,
\cite[\S 10]{IMN07}. In some cases, we will work with this version
of the inductive condition. We also note that if
Conjecture~\ref{thm:simple} holds for a quasisimple group $S$ (with
or without cyclic center), then it also holds for every quotient of
$S$ by a central subgroup. For further discussion, we refer the
reader to \cite[Remark~2.9]{CS}.

\begin{notation} Let $G$ be a finite group.
\begin{enumerate}[\rm(i)]
\item $d(G):=\min\{\chi(1):\chi\in\Irr(G), \chi(1)>1\}$ is the smallest nontrivial
degree of a (complex) irreducible character of $G$. By convention,
$d(G)=1$ if $G$ is abelian.
\item $m_p(G):=\min\{\chi(1):\chi\in\Irr_{p'}(G), \chi(1)>1\}$ is the smallest nontrivial
degree of an irreducible $p'$-character of $G$. By convention,
$m_p(G)=1$ if $G$ has no non-linear $p'$-degree irreducible
character.
\item $b_p(G):=\max\{\chi(1):\chi\in\Irr_{p'}(G)\}$ is the largest
degree of an irreducible $p'$-character of $G$.
\item We will use $m(G)$ and $b(G)$, respectively, for $m_p(G)$ and
$b_p(G)$, whenever $p$ is implicitly known or the presence of $p$ is
not important.
\end{enumerate}
\end{notation}

\begin{hypothesis}
\label{hypo} $m(S)\geq b(\bN_S(P))$, where $S$ is a quasisimple
group, $p$ is a prime, and $P\in\Syl_p(S)$.
\end{hypothesis}

\begin{pro}\label{pro:key}
Fix a quasisimple group $S$ and a prime $p$. If
Hypothesis~\ref{hypo}, or the stronger condition $d(S)\geq
b(\bN_S(P))$, is true for $S$, then so is
Conjecture~\ref{thm:simple} for $S$.
\end{pro}

\begin{proof}
If $m(S)\geq b(\bN_S(P))$, then any bijection $\Psi$ from
$\Irr_{p'}(S)$ to $\Irr_{p'}(\bN_S(P))$ sending $1_S$ to
$1_{\bN_S(P)}$ automatically satisfies the degree condition
$\Psi(\chi)(1)\leq\chi(1)$ for every $\chi\in \Irr_{p'}(S)$.
\end{proof}

Hypothesis~\ref{hypo}, unfortunately, fails quite often. For
instance, in general, it fails when $S$ is a cover of an alternating
or a simple classical group (in characteristic not equal to $p$). It
also fails for certain groups of Lie type in characteristic $p$ (see
the proof of Proposition~\ref{prop:defining-char}).

%%%%%%%%%%%%%%%%%%%%%%%%%%%%%%%%%
%%%%%%%%%%%%%%%%%%%%%%%%%%%%%%%%%

%%%%%%%%%%%%%%%%%%%%%%%%%%%%%%%%%%
%%%%%%%%%%%%%%%%%%%%%%%%%%%%%%%%%%

\subsection{Groups of Lie type in characteristic $p$}

The failure of Hypothesis~\ref{hypo} when $S$ is a group of Lie type
in characteristic $p$ arises in the case of unitary groups. This
case requires additional work. Our notation for finite simple groups
(and related ones) follows \cite{Carter85,Atlas}.

Low-degree irreducible representations of the special unitary groups
$SU_n(q)$ ($q=p^f$ is a power of a prime $p$ and $n\geq 3$,
excluding $(n,q)=(3,2)$) are studied in \cite[\S4]{TZ96} and
\cite[\S6.1]{LOST10}. Among these are the so-called irreducible Weil
characters, denoted by $\zeta^i_{n,q}$ for $0\leq i\leq q$. The
characters $\zeta^i_{n,q}$ with $i>0$ have degree
$(q^n-(-1)^n)/(q+1)$, while $\zeta^0_{n,q}$ has degree
$(q^n+q(-1)^n)/(q+1)$. In fact, these $\zeta^i_{n,q}$ account for
all nontrivial characters of $SU_n(q)$ with degree at most
$d(SU_n(q))+1$. Consequently, if $p$ is the defining characteristic
of the group, we have
\[m_p(SU_n(q))=(q^n-(-1)^n)/(q+1).\] For our purposes, we
need to construct a non-Weil character of the unitary groups that is
invariant under the automorphism groups.

To do so, we first describe the Weil characters of $G:=SU_n(q)$ via
the notion of Lusztig series and semisimple characters, as follows.
The set $\Irr(G)$ of irreducible characters of $G$ is partitioned
into the Lusztig series $\mathcal{E}(G,s)$, where $s$ runs over a
complete set of representatives of conjugacy classes of semisimple
elements of $G^*:=PGU_n(q)$ (see \cite[Theorem~2.6.2]{GM20}). Each
$\mathcal{E}(G,s)$ contains one or more special members called
semisimple characters whose degrees are equal to
$|G^*:\bC_{G^*}(s)|_{p'}$ (see \cite[Definition~2.6.9]{GM20}).

Centralizers of semisimple elements in classical groups are well
known (see \cite{Carter81,Fra20} for instance). We recall here the
needed result for $GU_n(q)$.

\begin{lem}
\label{lem:cen} Let $G=GU_n(q)$ and $s\in G$ be a semisimple
element. For a monic polynomial
$g(t)=t^d+a_{d-1}t^{d-1}+\cdots+a_1t+a_0\in \FF_{q^2}(t)$ not equal
to $t$, write
$g^*(t):=t^d+(a_1/a_0)^qt^{d-1}+\cdots+(a_{d-1}/a_0)^qt+(1/a_0)^q$.
Let $f(t)$ be the characteristic polynomial of $s$ and assume that
its decomposition into irreducible monic polynomials over
$\FF_{q^2}$ is
\[
f(t)=\pm\prod_{i=1}^a f_i(t)^{n_i} \times \prod_{j=1}^b
(g_j(t)g_j^*(t))^{m_j},
\]
where $f_i=f_i^*$ and $g_j\neq g_j^*$ for every $i$ and $j$ and they
are pairwise distint. Let $d_i:=\deg(f_i)$ and $e_j:=\deg(g_j)$.
Then
\[
\bC_G(s)\cong \prod_{i=1}^a GU_{n_i}(q^{d_i})\times \prod_{j=1}^b
GL_{m_j}(q^{2e_j}).
\]
\end{lem}

\begin{remark}
Irreducible factors of the characteristic polynomial $f(t)$ of a
semisimple element in $GU_n(q)$ are subject to certain restrictions.
For example, a polynomial $g$ is a factor if and only if $g^*$ is
also a factor. Additionally, any factor $g$ satisfying $g=g^*$ must
have odd degree.
\end{remark}

Let $\delta$ be a generator of the multiplicative group
$\FF_{q^2}^\times$. For $1\leq i\leq q$, consider the semisimple
element $s_i\in G^*$ to be the image of the diagonal matrix
\[\widetilde{s}_i:=\diag(\delta^{i(q-1)},1^{n-1})\in
\widetilde{G}:=GU_n(q)\] under the natural projection
$\pi:\widetilde{G}\to G^*$. For general $s$, $\mathcal{E}(G,s)$ may
contain more than one semisimple character, but we claim that, each
$\mathcal{E}(G,s_i)$ contains a unique one.

%This is done by embedding $SU_n(q)$ into $\widetilde{G}$. (See
%\cite[\S1.7]{GM20} for the background on this so-called regular
%embedding.)

Note that $\widetilde{G}$ is self-dual (in the sense of
\cite[\S4.3]{Carter85}) and we may identify $\widetilde{G}$ with its
dual group. As the center of the ambient algebraic group of
$\widetilde{G}$ is connected, the series
$\mathcal{E}(\widetilde{G},\widetilde{s}_i)$ contains a unique
semisimple character (\cite[p.~171]{GM20}. We denote this character
by $\psi_i$. Then
$\psi_i(1)=|\widetilde{G}:\bC_{\widetilde{G}}(\widetilde{s}_i)|_{p'}$.
By Lemma~\ref{lem:cen},
\[
\bC_{\widetilde{G}}(\widetilde{s}_i) \cong GU_{n-1}(q)\times
GU_1(q).
\]
The choice of eigenvalues of $s_i$'s implies that
$\bC_{G^*}(s_i)=\bC_{\widetilde{G}}(\widetilde{s}_i)/\bZ(\widetilde{G})$.
(To see this, let $C:=\pi^{-1}(\bC_{G^*}(s_i))$. Then the mapping
$\tau: C\to \bZ(\widetilde{G})$ defined by $gsg^{-1}=\tau(g)s$ is a
homomorphism with $\Ker(\tau)=\bC_{\widetilde{G}}(\widetilde{s}_i)$.
However, the fact that $s$ and $gsg^{-1}$ have the same eigenvalues
forces $\tau(g)=1$ for every $g\in C$.) It follows that
\[
|G^*:\bC_{G^*}(s_i)|_{p'}=|\widetilde{G}:\bC_{\widetilde{G}}(\widetilde{s}_i)|_{p'}=\frac{q^n-(-1)^n}{q+1},
\]
so that semisimple characters in $\mathcal{E}(G,s_i)$ have the same
degree as the (only) semisimple character $\psi_i$ in
$\mathcal{E}(\widetilde{G},\widetilde{s}_i)$. By
\cite[Corollary~2.6.18]{GM20}, semisimple characters in
$\mathcal{E}(G,s_i)$ are irreducible constituents of the restriction
of $\psi_i$ from $\widetilde{G}$ to $G^*$. We deduce that
$\mathcal{E}(G,s_i)$ has a unique semisimple character, as claimed.
This is the Weil character $\zeta^i_{n,q}$ mentioned above.

\begin{lem}
\label{lem:Weil-character} Assume the above notation. Then a Weil
character $\zeta^i_{n,q}$ for some $1\leq i\leq q$ is invariant
under all field automorphisms of $G$ if and only if $i(p-1)$ is
divisible by $q+1$. In particular, if $p=2$, then every
$\zeta^i_{n,q}$ is moved by some automorphism of $G$.
\end{lem}

\begin{proof}
The character $\zeta^i_{n,q}$ is invariant under all field
automorphisms of $G$ if and only if
$\delta^{i(q-1)}=\delta^{ip(q-1)}$, which means that $i(q-1)\equiv
ip(q-1) (\bmod\,q^2-1)$, and the first statement follows. When
$p=2$, we then have that each $\zeta^i_{n,q}$ (for $1\leq i\leq q$)
is moved by some field automorphisms of $G$, as desired.
\end{proof}

\begin{lem}
\label{lem:SU} Let $G=SU_n(q)$ where $n\geq 3$ is odd, $(n,q+1)=1$,
and $q=p^f$ for some prime $p$. Then $G$ possesses an irreducible
character $\chi$ of degree not divisible by $p$ such that $\chi$ is
$\Aut(G)$-invariant and $\chi(1)>q^{n-1}$.
\end{lem}

\begin{proof}
By the hypothesis, $G$ is simple and self-dual. By
\cite[Lemma~5]{Brunat09}, $p'$-degree irreducible characters of $G$
are precisely the semisimple characters of $G$, which in turn can be
labeled by conjugacy classes of semisimple elements of
$G^*:=PGU_n(q)\cong G$.

It is well known that $\Aut(G)$ permutes the Lusztig series of $G$.
In our situation, by identifying the automorphisms of $G$ with the
corresponding automorphisms of $G^*$ under the natural isomorphism,
every $\varphi\in\Aut(G)$ maps $\mathcal{E}(G,s)$ to
$\mathcal{E}(G,\varphi(s))$ (see \cite[Proposition~7.2]{Taylor18}).
Since each $\mathcal{E}(G,s)$ contains a unique semisimple character
(using again \cite[p.~171]{GM20} and the fact that the ambient
algebraic group of $PGU_n(q)$ has connected center), of degree
$|G^*:\bC_{G^*}(s)|_{p'}$, it follows that a $p'$-degree irreducible
characters of $G$ is $\Aut(G)$-invariant if and only if the
semisimple conjugacy class labeling it is $\Aut(G^*)$-invariant.
Therefore, the result follows if we are able to produce a semisimple
element $s\in G^*$ such that its $G^*$-conjugacy class is
$\Aut(G^*)$-invariant and $|G^*:\bC_{G^*}(s)|_{p'}>q^{n-1}$. Recall
that $\Out(G^*)\cong C_{2f}$ is the cyclic group of order $2f$
consisting of the field automorphisms of $G^*$.

Suppose first that $q$ is odd. Let $s$ be the image of a diagonal
matrix $\widetilde{s}\in \widetilde{G}:=GU_n(q)$ with spectrum
$\Spec(\widetilde{s})=\{1,\ldots,1,-1,-1\}$ under the projection
$\pi:\widetilde{G}\to G^*$ mentioned earlier. Note that $n$ is odd.
Similar arguments as above show that
$\bC_{\widetilde{G}}(\widetilde{s})$ is the complete inverse image
of $\bC_{G^*}(s)$. Moreover,
$\bC_{\widetilde{G}}(\widetilde{s})\cong GU_{n-2}(q)\times GU_2(q)$.
It follows that
\[
|G^*:\bC_{G^*}(s)|_{p'}=|\widetilde{G}:\bC_{\widetilde{G}}(\widetilde{s})|_{p'}
=\frac{(q^n+1)(q^{n-1}-1)}{(q^2-1)(q+1)}.
\]
Observe that, as $\widetilde{s}$ is invariant under all field
automorphisms, $s$ is $\Aut(G^*)$-invariant. The required character
$\chi$ can be taken to be the (only) semisimple character in the
Lusztig series associated to $s$.

%Assume now that $q$ is even, and furthermore, $n>q+1$. In
%particular, $n\geq 5$. Let
%$\Phi:=\{\lambda^{q-1}:\lambda\in\FF_{q^2}^\times\}$, so that $\Phi$
%consists of nonzero element in $\FF_{q^2}$ of order dividing $q+1$
%and $|\Phi|=q+1$. Choosing now $\widetilde{s}$ to be a diagonal
%matrix with $\Spec(\widetilde{s})=\Phi\cup \{1^{n-q-1}\}$ instead,
%we have $\bC_{\widetilde{G}}(\widetilde{s})\cong GU_{n-q-1}(q)\times
%GU_1(q)^{q+1}$, and
%\[
%|G:\bC_G(s)|_{p'}=\frac{\prod_{i=n-q}^n (q^i-(-1)^i)}{(q+1)^{q+1}}.
%\]
%It is straightforward to check that this is greater than $q^{n-1}$,
%and we conclude as in the previous case.

Assume now that $q=2^f$ with $f$ even. We then take $\widetilde{s}$
to be a diagonal matrix with
$\Spec(\widetilde{s})=\{\delta^{(q^2-1)/3},
\delta^{2(q^2-1)/3},1^{n-2}\}$ instead, where $\delta$, as before,
is a generator of $\FF_{q^2}^\times$. Again, by Lemma~\ref{lem:cen},
we have $\bC_{\widetilde{G}}(\widetilde{s})\cong GU_{n-2}(q)\times
GL_1(q^2)$ and $\bC_{\widetilde{G}}(\widetilde{s})$ is the complete
inverse image of $\bC_{G^*}(s)$. (This is clear when $n>3$. When
$n=3$, the image of the homomorphism $\tau:
\pi^{-1}(\bC_{G^*}(s))\to \bZ(\widetilde{G})$ defined by
$gsg^{-1}=\tau(g)s$ is contained in $\{\mathrm{Id},
\delta^{(q^2-1)/3}\mathrm{Id}, \delta^{2(q^2-1)/3}\mathrm{Id}\}$.
However, both $\delta^{(q^2-1)/3}\mathrm{Id}$ and
$\delta^{2(q^2-1)/3}\mathrm{Id}$ have order $3$, which does not
divide $|\bZ(\widetilde{G}|=q+1$. Therefore we still have
$\tau(g)=\mathrm{Id}$ for all $g\in \pi^{-1}(\bC_{G^*}(s))$.) We
then have
\[
|G^*:\bC_{G^*}(s)|_{p'}=\frac{(q^n+1)(q^{n-1}-1)}{(q^2-1)(q-1)}.
\]
It is straightforward to check that this is greater than $q^{n-1}$,
and we conclude as in the previous case.

We solve the remaining case $q=2^f$ with $f$ odd by a somewhat
different argument, working with the ambient algebraic group of $G$
and its associated Frobenius map instead. So let $\bG$ be a simple
algebraic group of adjoint type of type $A$, defined over an
algebraically closed field of characteristic $2$. Let $F_f$ be the
standard Frobenius map on $\bG$, raising all matrix entries to the
$2^f$-th power and $\rho$ the inverse transpose. Set
$F:=F_f\circ\rho$. Then $G=G^*$ is precisely the group of $F$-fixed
points in $\bG$. Let $F_0:=F_1\circ\rho$. As $f$ is odd, we have
$F=F_f\circ\rho=F_0^f$. Note that $F_1$ induces a generator, say
$\sigma_{F_1}$, for $\Out(G)\cong C_{2f}$; on the other hand, $F_0$
induces $\sigma_{F_0}\in\Out(G)$ of order $f$ and $\langle
\sigma_{F_0}\rangle=\langle (\sigma_{F_0})^2\rangle=\langle
(\sigma_{F_1})^2\rangle$. By the work of Brunat
\cite[Proposition~2]{Brunat09}, the number of
$\sigma_{F_0}$-invariant irreducible characters of $G$ of odd degree
is equal to the number of semisimple characters of
$\bG^{F_0}=PGU_n(2)\cong SU_n(2)$. (Note that, by the assumptions,
$f$ is odd and $(n,2^f+1)=1$. In particular, $n$ is coprime to $3$.)
This number, in turn, is $2^{n-1}$, by
\cite[Corollary~8.3.6]{Carter85}. These $2^{n-1}$ characters are
permuted by $\sigma_{F_1}$, forming orbits of size $1$ or $2$. The
trivial character forms its own orbit. Consequently, at least one of
the nontrivial characters, say $\chi$, is fixed by $\sigma_{F_1}$ --
that is, $\chi$ is $\Out(G)$-invariant. By
Lemma~\ref{lem:Weil-character}, $\chi$ cannot be a Weil character.
Moreover, under the assumptions on $n$, $q$, and $f$, we have $n
\geq 5$. It then follows from \cite[Table~V]{TZ96} that $\chi(1)\geq
(q^n+1)(q^{n-1}-q^2)/(q+1)(q^2-1)$, and therefore $\chi(1)>q^{n-1}$.
This concludes the proof.
\end{proof}

\begin{lem}\label{lem:11}
Let $G = \bG^F$ be the group of fixed points of a simple algebraic
group $\bG$ of simply connected type over an algebraically closed
field of characteristic $p$, under a Steinberg endomorphism $F$ of
$\bG$, and assume that $G/\bZ(G)$ is simple. Let $P \in \Syl_p(G)$.
Then
\[
b\bigl(\bN_G(P)\bigr) \leq |T/\bZ(G)|,
\]
where $T := \bT^F$ and $\bT$ is a maximally split $F$-stable maximal
torus of $\bG$.
\end{lem}

\begin{proof}
Let $\bB$ be an $F$-stable Borel subgroup of $\bG$ containing $\bT$.
Let $\Phi$ be the root system of $\bG$ with respect to $\bT$ and
$\bB$. Let $\Phi^+$ and $\Delta=\{\alpha_i:i\in I\}$ be the
corresponding set of positive roots and simple roots, respectively.
Also, let $\bX_\alpha$ denote the root subgroup associated to each
$\alpha\in\Phi$.

Let $\bU:=\prod_{\alpha\in \Phi^+}\bX_{\alpha}$, which is the
unipotent radical of $\bB$. According to
\cite[Corollary~24.11]{MT11}, we have $P:=\bU^F\in\Syl_p(G)$ and
$\bB^F=\bN_G(P)$. Furthermore, \[\bN_G(P)=P\rtimes T,\] where
$T:=\bT^F$. At this point we have $b(\bN_G(P))\leq |\bN_G(P)/P|=|T|$
but, in many cases, this upper bound for $b(\bN_G(P))$ is not
sufficient. We need to know more details about the action of $T$ on
$\Irr(P/P')$.

Let $\bU_c:=\prod_{\alpha\in \Phi^+\setminus\Delta} \bX_{\alpha}$.
As explained in \cite[\S2.9]{Carter85}, $\bU_c$ is normal in $\bU$
and $\bU/\bU_c=\prod_{i\in I}\bX_{\alpha_i}$. The endomorphism $F$
naturally acts on the roots, and thus on the root subgroups, given
by $F(\bX_\alpha)=\bX_{F(\alpha)}$. As $\bB$ and $\bT$ are
$F$-stable, $F$ permutes the positive roots, as well as the simple
roots. Both $\bU$ and $\bU_c$ are therefore $F$-stable. Let $\rho$
denote the permutation on $I$ induced from the action of $F$ on the
simple roots, and $\mathcal{O}$ be the set of $\rho$-orbits on $I$.
For each such an orbit $J$, let $\bX_J:=\prod_{i\in J}
\bX_{\alpha_i}$, which is an $F$-stable group. Further, let
$X_J:=\bX_J^F$. We then have
\begin{equation}\label{eq:1}\bU^F/{\bU_c}^F= \prod_{J\in \mathcal{O}}
X_J.\end{equation}

Note that each $\bX_\alpha$ is normalized by $\bT$ (see
\cite[p.~18]{Carter85}). Thus $\bU^F$ and ${\bU_c}^F$ are both
normalized by $T=\bT^F$, and so $T$ acts on the factor group
$\bU^F/{\bU_c}^F$. In fact, each direct factor $X_J$ with $J\in
\mathcal{O}$ of $\bU^F/{\bU_c}^F$ is normalized by $T$. %Note that
%$X_J$ is isomorphic to $\FF_{q^{|J|}}$ and $|X_J|=q^{|J|}$, where
%$q$ is the absolute value of all eigenvalues of $F$ on the character
%group of an $F$-stable maximal torus of $\bG$. Furthermore, as noted
%in \cite[p.~74]{Carter85}, we have
%\[
%|T|=\prod_{J\in \mathcal{O}} (q^{|J|}-1).
%\]
Remark that $\bZ(G)\leq T$ (see the proof of
\cite[Lemma~24.12]{MT11}) and $\bZ(G)$ acts trivially on $\bU^F$.
Therefore the maximal size of a $T$-orbit on $\bU^F/{\bU_c}^F$ is at
most $|T/\bZ(G)|$.

Assume from now on that $G\notin \{Sp_{2n}(2),F_4(2),G_2(3)\}$. (For
these exceptions, a Sylow $2$-subgroup of $G$ is self-normalizing
and the desired inequality is immediate.) Then, according to
\cite[Lemma~5]{Brunat09}, ${\bU_c}^F$ is the derived subgroup $P'$
of $P=\bU^F$. As the actions of $T$ on $P/P'$ and $\Irr(P/P')$ are
isomorphic, $|T/\bZ(G)|$ is also an upper bound for the sizes of the
$T$-orbits on $\Irr(P/P')$. Moreover, the restriction of every
irreducible character of $P/P'\rtimes T$ to $P/P'$ is
multiplicity-free, as $T$ is abelian and $|T|$ is coprime to $p$.
The result now follows.
\end{proof}

\begin{lem}\label{lem:12}
Assume the hypotheses of Lemma~\ref{lem:11}. Then $S=G/\bZ(G)$
satisfies Hypothesis~\ref{hypo}, unless $G=SU_n(q)$ with $n\geq 3$
odd and $(n,q+1)=1$.
\end{lem}

\begin{proof}
%By the hypothesis, $S$ is a quotient of $G=\bG^F$, the group of
%fixed points of a simple algebraic group $\bG$ of simply-connected
%type over an algebraically closed field of characteristic $p$ under
%a Steinberg endomorphism $F$ of $\bG$.
Clearly, $m(G)\leq m(S)$ and
$b(\bN_G(Q))\geq b(\bN_S(P))$ for $P\in\Syl_p(S)$ and $Q\in
\Syl_p(G)$. Therefore, if Hypothesis~\ref{hypo} holds for $G$, then
it also holds for $S$. Sylow $p$-subgroups and their normalizers of
a finite reductive group in characteristic $p$ are best described
through the framework of Borel subgroups and their unipotent
radicals in the ambient algebraic group. For convenience, we shall
use the same $P$ for a Sylow $p$-subgroup of $G$ (in fact, $P\cong
Q$, see \cite[p.~214]{MT11} ).

Note that the group $X_J$ in the proof of Lemma~\ref{lem:11} is
isomorphic to $\FF_{q^{|J|}}$ and $|X_J|=q^{|J|}$, where $q$ is the
absolute value of all eigenvalues of $F$ on the character group of
an $F$-stable maximal torus of $\bG$. Furthermore, as noted in
\cite[p.~74]{Carter85}, we have
\[
|T|=\prod_{J\in \mathcal{O}} (q^{|J|}-1).
\]
Lemma~\ref{lem:11} therefore implies that
\[
b(\bN_G(P))\leq \frac{\prod_{J\in \mathcal{O}}
(q^{|J|}-1)}{|\bZ(G)|}.
\]

%In most cases, $({\prod_{J\in \mathcal{O}} (q^{|J|}-1)})/{|\bZ(G)|}$
%is bounded above by $m(G)$. There are, however, some exceptions that
%we have to examine individually.

The $\rho$-action on the set of the simple roots, defined in the
proof of Lemma~\ref{lem:11}, for all the relevant $(\bG,F)$ is
described in \cite[p.~37]{Carter85}, allowing one to easily
determine the sizes of $\rho$-orbits. Generally, $|J|\in\{1,2,3\}$
for all $(\bG,F)$ and when $G$ is untwisted, $|J|=1$ for all $J$. On
the other hand, the values of the smallest nontrivial $p'$-degree
$m(G)$ can be read off from \cite{TZ96} and \cite{Nguyen10} when $G$
is of classical type and those for exceptional types can be found in
\cite{Lubeck01}.

Consider $G=SL_2(q)$ with $q$ odd. Then $|T|=q-1$ and so
$b(\bN_G(P))\leq (q-1)/2$. On the other hand, $m(G)=(q-1)/2$ and so
Hypothesis~\ref{hypo} is satisfied. (In this case, in fact,
$b(\bN_G(P))=m(G)=(q-1)/2$. Each element $t$ of $T\cong
\FF_q^\times$ acts on $P\cong \FF_q^+$ by mapping $u$ to $ut^2$, and
so the stabilizer in $T$ of any nontrivial (linear) character of $P$
is precisely the order-2 subgroup of $T$. Therefore $\bN_G(P)$ has
$q-1$ linear characters and four characters of degree $(q-1)/2$.)
Similarly, for $G=SL_2(q)$ with $q$ even, one has $m(G)=q-1=|T|$ and
Hypothesis~\ref{hypo} is still valid.

Let $G=SL_n(q)$ with $n\geq 3$. Here $|T|=(q-1)^{n-1}$ and
$m(G)=(q^n-1)/(q-1)$, unless $(n,q)=(4,3)$. Note that
$m(SL_4(3))=26$. In any case, we have $|T|\leq m(G)$, and
Hypothesis~\ref{hypo} is verified.

Let $G=SU_n(q)$ with $n\geq 3$. Then
\begin{equation*}
|T| = \begin{cases}
(q^2-1)^{(n-1)/2} &\text{ if } n \text{ is odd}\\
(q^2-1)^{(n-2)/2}(q-1) &\text{ if } n \text{ is even}.
\end{cases}
\end{equation*}
and, as mentioned already, \[m(G)=(q^n-(-1)^n)/(q+1).\] We now can
verify that Hypothesis~\ref{hypo} holds when $n$ is even or $n$ is
odd and $|\bZ(G)|=(n,q+1)> 1$.

For $G=Sp_{2n}(q)$ or $Spin_{2n+1}(q)$ with $n\geq 2$, we have
$|T|=(q-1)^n$ and $b(\bN_G(P))\leq (q-1)^n/(2,q-1)$. One can verify
from Table~\ref{table} that, $d(G)$, which serves as a lower bound
for $m(G)$, is always at least $(q-1)^n/(2,q-1)$. Similarly, for
$G=Spin_{2n}^\epsilon(q)$ with $\epsilon\in\{\pm\}$ and $n\geq 4$,
we have $b(\bN_G(P))\leq
(q-1)^{n-1}(q-\epsilon1)/(4,q^n-\epsilon1)$, while
$m(G)\geq(q^n-1)(q^{n-1}-1)/2(q+1)$, by \cite[Theorems~1.3 and
1.4]{Nguyen10}. Hypothesis~\ref{hypo} is again satisfied.

Let $G=\ta B_2(q^2)$ or $\ta G_2(q^2)$, where $q^2=2^{2n+1}$ or
$q^2=3^{2n+1}$, respectively. In these cases, $|T|=q^2-1$, which is
less than $m(G)$. (For type $\ta B_2$, $m(G)$ is at least
$\sqrt{1/2}q(q^2-1)$, while for type $\ta G_2$, $m(G)=q^4-q^2+1$.)
Similarly, when $G=\ta F_4(q^2)$, we have $|T|=(q^2-1)^2$ and
$m(G)\geq \sqrt{1/2}q^9(q^2-1)$. In all these cases, the required
inequality holds.

For the remaining exceptional-type groups, we always have $|T|\leq
q^{rk(G)}$, where $rk(G)$ is the semisimple rank of $\bG$. However,
the lower bound for the smallest nontrivial irreducible
representation of $G$, as shown in Table~\ref{table}, confirms that
$m(G)>q^{rk(G)}$.
\end{proof}

\begin{pro}
\label{prop:defining-char} Conjecture~\ref{thm:simple} is true when
$S$ is a quotient of a non-exceptional covering group of a simple
group of Lie type in characteristic $p$.
\end{pro}

\begin{proof}
We assume that $S$ is not the Tits group ${}^2F_4(2)'$, since
Hypothesis~\ref{hypo} can be verified directly for this group using
\cite{GAP}. By Lemma~\ref{lem:12} and Proposition~\ref{pro:key}, it
remains to consider only the case where $G = SU_n(q)$ with $n \geq
3$ odd and $(n, q+1) = 1$.

In such case, $G$ can be viewed as a group of adjoint type and
therefore, by \cite[\S2.9]{Carter85}, \[T\cong\prod_{J\in
\mathcal{O}} T_J,\] the direct product of cyclic groups $T_J\cong
C_{q^2-1}$. (Here, every orbit $J$, defined in the proof of
Lemma~\ref{lem:11}, has length $2$.) Further, the action of $T$ on
$P/P'\cong \prod_{J\in \mathcal{O}} X_J$ (see \eqref{eq:1}) is a
`product' action, in the way that $T_J$ acts trivially on $X_{J'}$
if $J\neq J'$ and transitively on $X_J \setminus \{1\}$. Therefore
$T$ has a (unique) regular orbit on $P/P'$, as well as on
$\Irr(P/P')$, implying that $b(\bN_G(P))=|T|$ and $\bN_G(P)$ has a
unique $p'$-degree irreducible character of degree
$|T|=(q^2-1)^{(n-1)/2}$. We shall denote this character by $\tau$.
Every other $p'$-degree irreducible character of $\bN_G(P)$
restricts trivially to at least one of $X_J$'s, and hence has degree
at most $|T|/(q^2-1)$.

Note that, as $(n,q+1)=1$, $G$ is simple and the group, say $A$, of
outer automorphisms of $G$ is cyclic and stabilizes (the unipotent
subgroup) $P$. Moreover, a bijection between $\Irr_{p'}(G)$ and
$\Irr_{p'}(\bN_G(P))$ is McKay-good if and only if it is
$A$-equivariant. The existence of such a bijection was established
in \cite{Brunat09,Spath12}. In other words, we know that, for every
subgroup $B\leq A$, the numbers of $B$-fixed characters in
$\Irr_{p'}(G)$ and $\Irr_{p'}(\bN_G(P))$ are the same. Note that the
above-mentioned character $\tau$ of $\bN_G(P)$ is $A$-invariant (due
to its uniqueness property). Let $\xi$ be the character of $G$
produced by Lemma~\ref{lem:SU}; in particular, $\xi$ is $p'$-degree
and $A$-invariant. Now the numbers of $B$-fixed characters in
$\Irr_{p'}(G)\setminus\{\xi\}$ and
$\Irr_{p'}(\bN_G(P))\setminus\{\tau\}$ are the same for every $B\leq
A$. Using \cite[Lemma~13.23]{Is}, one can construct an
$A$-equivariant bijection from $\Irr_{p'}(G)\setminus\{\xi\}$ to
$\Irr_{p'}(\bN_G(P))\setminus\{\tau\}$, which can be extended to an
$A$-equivariant bijection $\Psi:\Irr_{p'}(G)\to\Irr_{p'}(\bN_G(P))$
such that $\Psi(\xi)=\tau$.

We claim that $\chi(1)\geq \Psi(\chi)(1)$ for all $\chi\in
\Irr_{p'}(G)$. First, observe that
\[\xi(1)>q^{n-1}>(q^2-1)^{(n-1)/2}=\tau(1).\] On the other hand,
when $\chi\neq\xi$, we have $\Psi(\chi)\neq \tau$, and so
$\Psi(\chi)(1)\leq |T|/(q^2-1)=(q^2-1)^{(n-3)/2}$, which implies
that
\[
\chi(1)\geq (q^n-(-1)^n)/(q+1)\geq \Psi(\chi)(1),
\]
and this finishes the proof.
\end{proof}

%%%%%%%%%%%%%%%%%%%%%%%%%%%%%%%%%%%%%%%%%%%%%%%%%%%%%%%%%%

\begin{table}[ht]
\caption{Values and/or bounds for the minimal nontrivial degree of
ordinary characters of finite reductive groups of simply-connected
type \cite{TZ96,Lubeck01}.\label{table}}
\begin{center}
\begin{tabular}{lll}
\hline $G$ & \begin{tabular}{l} Conditions\end{tabular} & \begin{tabular}{l}$d(G)$\end{tabular} \\
\hline

$SL_2(q)$& \begin{tabular}{l} $q\geq
5$\end{tabular}&\begin{tabular}{l}${(q-1)}/{(2,q-1)}$\end{tabular}\\\hdashline

$SL_n(q)$& \begin{tabular}{l} $n\geq 3$\\
$(n,q)\not\in\{(3,2),(3,4),(4,2),(4,3)\}$\end{tabular} &
\begin{tabular}{l}${(q^n-1)}/{(q-1)}$\end{tabular}
\\\hdashline

$SU_n(q)$ & \begin{tabular}{l} $n\geq 3$ odd; $(n,q)\neq (3,2)$\\
$n\geq 4$ even; $(n,q)\not\in\{(4,2),(4,3)\}$\end{tabular}  & \begin{tabular}{l}${(q^n-q)}/{(q+1)}$ \\
 ${(q^n-1)}/{(q+1)}$\end{tabular}\\\hdashline

$Sp_{2n}(q)$& \begin{tabular}{l} $n\geq 2$; $q$ odd\\
$n\geq 2$; $q$ even; $(n,q)\neq(2,2)$ \end{tabular}& \begin{tabular}{l} ${(q^n-1)}/{2}$\\
${(q^n-1)(q^n-q)}/{2(q+1)}$ \end{tabular} \\\hdashline

$Spin_{2n+1}(q)$& \begin{tabular}{l}$n\geq 3$; $q>3$ odd\\
$n\geq 3$; $q=3$
\end{tabular} &
\begin{tabular}{l}${(q^{2n}-1)}/{(q^2-1)}$\\
$(q^n-1)(q^n-q)/(q^2-1)$ \end{tabular}
\\\hdashline

$Spin^+_{2n}(q)$& \begin{tabular}{l} $n\geq 4$; $q>3$\\
$n\geq 4$; $q\in\{2,3\}$; $(n,q)\neq (4,2)$ \end{tabular} &
\begin{tabular}{l}$(q^n-1)(q^{n-1}+q)/(q^2-1)$\\
$(q^n-1)(q^{n-1}-1)/(q^2-1)$ \end{tabular}
\\\hdashline

$Spin^-_{2n}(q)$& \begin{tabular}{l}$n\geq 4$ \end{tabular} &
\begin{tabular}{l}$(q^n+1)(q^{n-1}-q)/(q^2-1)$\end{tabular}\\\hdashline

$\ta B_2(q^2)$& \begin{tabular}{l}$q^2=2^{2f+1}\geq 8$ \end{tabular}
&
\begin{tabular}{l}$\sqrt{1/2}q(q^2-1)$\end{tabular} \\\hdashline

$\ta G_2(q^2)$& \begin{tabular}{l}$q^2=3^{2f+1}\geq 27$\end{tabular}
&
\begin{tabular}{l}$q^4-q^2+1$\end{tabular}
\\\hdashline

$\ta F_4(q^2)$& \begin{tabular}{l}$q^2=2^{2f+1}\geq 8$ \end{tabular}
&
\begin{tabular}{l}$\sqrt{1/2}q^9(q^2-1)$\end{tabular}\\\hdashline

$G_2(q)$& \begin{tabular}{l}$q\geq 3$ \end{tabular} &
\begin{tabular}{l}$\geq q^3-1$\end{tabular}
\\\hdashline

$\tb D_4(q)$&  & \begin{tabular}{l}$\geq q^3(q^2-1)$\end{tabular}
\\\hdashline

$F_4(q)$&  & \begin{tabular}{l}$\geq q^8+q^4+1$\end{tabular}
\\\hdashline

$E_6(q)_{sc}$&  & \begin{tabular}{l}$\geq q^9(q^2-1)$\end{tabular}
\\\hdashline

$\ta E_6(q)_{sc}$&  & \begin{tabular}{l}$\geq
q^9(q^2-1)$\end{tabular}
\\\hdashline

$E_7(q)_{sc}$&  & \begin{tabular}{l}$\geq
q^{15}(q^2-1)$\end{tabular}
\\\hdashline

$E_8(q)$&  & \begin{tabular}{l}$\geq q^{27}(q^2-1)$ \end{tabular} \\

\hline
\end{tabular}
\end{center}
\end{table}
%%%%%%%%%%%%%%%%%%%%%%

\subsection{Groups of Lie type in characteristic not equal to $p$}

\begin{lem}
\label{lem:1} Let $H\leq G$ be such that $|G:H|$ is not divisible by
$p$. Then $b(H)\leq b(G)$.
\end{lem}

\begin{proof}
Let $\varphi\in\Irr_{p'}(H)$ such that $\varphi(1)=b(H)$. By
Frobenius reciprocity, $\varphi$ is contained in $\chi_H$ for any
irreducible constituent $\chi$ of $\varphi^G$. Since
$\varphi^G(1)=\varphi(1)|G:H|$ is not divisible by $p$, at least one
of those constituents has $p'$-degree.
\end{proof}

%\begin{lemma}
%Suppose that the power of $\Phi_e$ in the order polynomial of $\bG$
%is 1. Then
%\[
%|W_{\bG}(\bL_e)|\leq \Phi_e(q)-1.
%\]
%\end{lemma}
%
%\begin{proof}
%From the hypothesis we know that $S_e$ is cyclic of order
%$\Phi_e(q)$ and $S_e$ contains a Sylow $p$-subgroup of $G$. Each
%element $\bN_G(\bS_e)$ induces an automorphism on $S_e$ by
%conjugation, and the mapping $\tau:\bN_G(\bS_e)\to \Aut(S_e)$ given
%by $\tau(g)(a):=a^{g}$ is a homomorphism.
%
%Let $g\in\Ker(\tau)$. Write $g=g_{s}$
%\end{proof}

\begin{pro}
\label{prop:cross-char} Conjecture~\ref{thm:simple} is true when $S$
is a non-exceptional covering group of a simple group of exceptional
Lie type in characteristic different from $p$.
\end{pro}

\begin{proof}
Assume for now that $S$ is not of Suzuki or Ree type. As before, $S$
is a quotient of $G:=\bG^F$, where $\bG$ is a simple algebraic group
of simply-connected type and $F$ is a Steinberg endomorphism on
$\bG$. Let $q$ be the absolute value of all eigenvalues of $F$ on
the character group of an $F$-stable maximal torus of $\bG$. We know
that $p$ does not divide $q$, and we may assume that $p$ divides
$|G|$.

We require some $d$-Harish-Chandra theory, particularly the concept
of Sylow $d$-tori, which was first introduced by Brou\'{e} and Malle
in \cite{BM92}. (For a detailed account, see also
\cite[\S3.5]{GM20}.) Define $e$ as the multiplicative order of $q$
modulo $p$ if $p>2$ or if $p=2$ and $q\equiv 1(\bmod\,4)$. For $p=2$
and $q\equiv -1(\bmod\,4)$, let $e:=2$. As $p\mid |G|$, we know that
$\Phi_e(q) \mid |G|$, where $\Phi_e$ is the $e$-th cyclotomic
polynomial. Let $\bS_e$ be a Sylow $e$-torus of $\bG$. It is, by
definition, an $F$-stable torus of $\bG$ whose order polynomial is
the maximal power of $\Phi_e$ dividing the generic order of $\bG$
(see \cite[p.~259]{GM20}).

Let $\bL_e:=\bC_\bG(\bS_e)$, known as a (minimal) $e$-split Levi
subgroup of $\bG$. It is $F$-stable (see \cite[p.~258]{GM20}) and we
write $L_e:=\bL_e^F$. Note that, by the conjugation property of
Sylow $d$-tori, we have $\bN_G(\bS_e)=\bN_G(\bL_e)$. The quotient
$W_\bG(\bL_e):=\bN_G(\bL_e)/\bL_e^F$ is referred to as the
\emph{relative Weyl group} of $\bL_e$ in $\bG$.

According to \cite[Proposition~5.21]{Malle07}, $\bN_G(\bS_e)$
contains a Sylow $p$-subgroup, say $P$, of $G$. In fact, by
\cite[Theorem~7.8]{Malle07}, $\bN_G(\bS_e)$ contains $\bN_G(P)$,
unless $p=3$ and $G = G_2(q)$ with $q\equiv2, 4, 5$, or $7
(\bmod\,9)$. For now, let us exclude this exception. It follows from
Lemma~\ref{lem:1} that
\[b(\bN_G(P))\leq b(\bN_G(\bL_e)).\]

Suppose first that $e$ is a regular number for $(\bG,F)$, which
means that $\bL_e$ is a maximal torus of $\bG$ (see
\cite[p.~259]{GM20}). Then $L_e:=\bL_e^F$ is abelian, and thus
$b(\bN_G(\bL_e))\leq |\bN_G(\bL_e)/L_e|$, by
\cite[Corollary~11.29]{Is}. In summary, if $e$ is a regular number
for $(\bG,F)$ and $G$ is not of type $G_2$, then
\[
b(\bN_G(P))\leq |W_\bG(\bL_e)|.
\]

Consider the case when the power of $\Phi_e$ in the order polynomial
of $\bG$ is precisely 1, or equivalently, where $W_\bG(\bL_e)$ is
cyclic (see \cite[Proposition~3.5.12]{GM20}). As computed in
\cite[Table~8.1]{BM93}, with one exception at type $E_8$ and $e=30$,
we have $|W_\bG(\bL_e)|\leq 24$ for all relevant $G$ and $e$. One
can easily check that the minimal character degree $d(G)$ of $G$,
displayed in Table~\ref{table}, is always at least $26$, and so we
are done. For the exception, we have $|W_\bG(\bL_e)|=30$, while
$d(G)\geq q^{27}(q^2-1)$ and so Hypothesis~\ref{hypo} is satisfied.
On the other hand, non-cyclic relative Weyl groups of minimal
$e$-split Levi subgroups for exceptional types are available in
\cite[Table~3.2]{GM20}. We have verified that $|W_\bG(\bL_e)|$
remains at most $d(G)$, except for the specific cases discussed in
the next paragraph.

Consider $G=E_8(q)$, for instance. Then $|W_\bG(\bL_e)|\leq d(G)$
unless $e=2$, $q=2$, and $p=3$. For the exception, $W_\bG(\bL_e)$ is
isomorphic to the Weyl group $W(E_8)=C_2.GO_8^+(2)$ of $E_8$, which
has order $696729600$, while the smallest nontrivial degree
$d(E_8(2))$ of $E_8(2)$ is $545925250$. However, since
$b_3(\bN_G(\bL_e))$ divides $|W_\bG(\bL_e)|$, by
\cite[Corollary~11.29]{Is}, we have $b_3(\bN_G(\bL_e))\leq
696729600_{3'}=2867200< d(G)$, and the result follows as before.
Other exceptions occur for $(G,q,p,e)=(E_7(q)_{sc},2,3,2)$ or $(\ta
E_6(2)_{sc},2,3,2)$, but the arguments are entirely similar.

Next we consider the case where $e$ is not a regular number for
$(\bG,F)$. This includes $e=5$ for type $E_6$; $e=10$ for $^2E_6$;
$e\in\{4,5,8,10,12\}$ for $E_7$; and $e\in\{7,9,14,18\}$ for $E_8$.
Note that $L_e$ is no longer abelian. But, as $S_e:=\bS_e^F$ is
abelian, we shall use the bound
\[
b(\bN_G(P))\leq b(\bN_G(\bL_e))\leq
|\bN_G(\bL_e):S_e|=|W_\bG(\bL_e)||L_e:S_e|
\]
instead.

Let $G=E_7(q)_{sc}$ and $e=4$. Then $|\bN_G(\bL_e)/L_e|=96$ and
$L_e$ has type $\Phi_4^2A_1(q)^3$ (see \cite[Table~3.3]{GM20}). Thus
$|\bN_G(\bL_e):S_e|\leq 96q^3(q^2-1)^3$, which is smaller than
$d(E_7(q)_{sc})$, as desired. In the remaining cases, $W_\bG(\bL_e)$
is cyclic and its order together with the structure of $L_e$ are
again given in \cite[Table~8.1]{BM93}. Note that, in this case, the
Sylow $e$-torus $S_e$ has order $\Phi_e(q)$. It is now
straightforward to check that $|W_\bG(\bL_e)||L_e:S_e|\leq d(G)$ for
all relevant $G$ and $e$.

For Suzuki and Ree groups, Brou\'{e} and Malle introduced an adapted
version of $\Phi_e$, denoted by $\Phi_e^{(p)}$, which are cyclotomic
polynomials over $\QQ(\sqrt{2})$ or $\QQ(\sqrt{3})$ (see
\cite[\S8]{Malle07}). For the Tits group, we can verify
Hypothesis~\ref{hypo} directly using \cite{GAP}; therefore, we
assume that $S\neq \ta F_4(2)'$. With this, there exists a Sylow
$\Phi_e^{(p)}$-torus $\bS$ of $\bG$ such that $\bN_G(\bS)\geq
\bN_G(P)$ for some Sylow $p$-subgroup $P$ of $G$, unless
\begin{enumerate}[\rm(i)]
\item $p=2$ and $G = \ta G_2(q)$, or
\item $p=3$ and $G=\ta F_4(q)$ with $q\equiv2, 5 (\bmod\, 9)$.
\end{enumerate}
(See \cite[Theorem~8.4]{Malle07}.) The case when such a Sylow
$\Phi_e^{(p)}$-torus exists can be argued similarly as above.
(Alternatively, one can use \cite[\S16 and \S17]{IMN07} and
\cite{An98} to achieve the result for these groups.) The case $p=2$
and $G = \ta G_2(q)$ follows from the proof of the McKay inductive
conditions for the group (see \cite[\S17]{IMN07}). When $p=3$ and
$\ta F_4(q)$, according to \cite[(2B)]{An98}, we have
$|\bN_G(P)/P|\leq 48$, and the same reasoning applies.

Finally, the case $p=3$ and $G_2(q)$ (that we previously excluded)
follows from An's proof of the Alperin-McKay conjecture for
$G_2(q)$. (See \cite[p.~190]{An94} where it was shown that
$|\bN_G(P)/P|$ is bounded above by $16$, which is smaller than
$d(G_2(q))$.)
\end{proof}

%%%%%%%%%%%%%%%%%%%%%%%%%%%%%%%%%

\subsection{Exceptional covering groups}

Here we deal with \emph{exceptional covering groups} of finite
simple groups. These include $3$-fold and $6$-fold covers of $\Al_6$
and $\Al_7$, covers of sporadic simple groups (by convention), and
certain covers of simple groups of Lie type with a non-generic Schur
multiplier (see \cite[Table~24.3]{MT11}). Here, a (perfect central)
cover $S$ of a simple group of Lie type $X$ is called exceptional if
it is not a quotient of the finite reductive group of
simply-connected type covering $X$; in particular, $S$ is a proper
cover of $X$.

\begin{pro}
\label{prop:1} Conjecture~\ref{thm:simple} is true when $S$ is an
exceptional covering group.
\end{pro}

\begin{proof}
The existence of a McKay-good bijection for exceptional covering
groups was established by Malle \cite{Malle08}. We note that, except
the single case $S=2\cdot PSL_3(4)$, the group of outer
automorphisms stabilizing $P\in\Syl_p(S)$ is either trivial or
cyclic of prime order, and so any bijection that respects the action
of those outer automorphisms is McKay-good. Building on Malle's
work, we show that in most cases, Hypothesis~\ref{hypo} holds,
ensuring that any such bijection automatically satisfies the
additional degree condition. In the remaining cases, it turns out
that $S$ has precisely one character of degree smaller than
$b(\bN_S(P))$. In these instances, it suffices to identify a
corresponding ($p'$-degree) character of $\bN_S(P)$ with smaller
degree that is invariant under the action of the outer
automorphisms.

Throughout we let $X:=S/Z$ and $Q:=PZ/Z\in\Syl_p(X)$, where
$Z:=\bZ(S)$. Of course we may assume that $p\mid |S|$.

First suppose that $|S|_p = p$; in particular, $p$ is odd. In this
case, work of Dade \cite{Dade66,Dade96} provides a natural bijection
between the irreducible characters in any block $B$ of $S$ and those
in its Brauer correspondent $b$, and this bijection satisfies the
required degree condition. We may assume that $B$ has full defect.
When $|S|_p=p$, Dade's bijection is described in Lemmas~4.7--4.10 of
\cite{Dade96} and is known, in particular, to preserve decomposition
numbers. Hence it suffices to verify the desired degree inequality
for the corresponding bijection between the Brauer characters of $B$
and $b$ given in \cite[Lemma~4.7]{Dade96}.

This bijection between Brauer characters is defined as follows. The
Green correspondence sends the isomorphism classes of
finite-dimensional non-projective indecomposable $kS$-modules $M$
belonging to $B$ bijectively onto the isomorphism classes of
finite-dimensional non-projective indecomposable $k\bN_S(P)$-modules
$\widetilde{M}$, where $k$ is a suitable residue field of
characteristic $p$ (see \cite[III.5]{Feit82}). Moreover, this
correspondence satisfies the degree inequality $\dim(M)\ge
\dim(\widetilde{M})$. If $M$ is simple, the socle $S(\widetilde{M})$
of $\widetilde{M}$ lies in a uniquely determined isomorphism class
of simple $k\bN_S(P)$-modules. Dade's bijection then sends the
Brauer character afforded by $M$ to that afforded by
$S(\widetilde{M})$, and therefore also satisfies the degree
inequality. (We also note that in a recent preprint \cite{Li25},
Linckelmann proved that, when a defect group of $B$ is cyclic in
general, there exists a perfect isometry between $\ZZ \Irr(B)$ and
$\ZZ \Irr(b)$ with the degree condition.) Furthermore, Koshitani and
Sp\"ath \cite{KS16} proved that Dade's bijection fulfills the
inductive Alperin--McKay condition, and hence is McKay-good. We
assume from now on that $p^2\mid |S|$.

\medskip

I) Consider the sporadic groups. The structure of the quotient group
$\bN_X(Q)/Q'$ is given in \cite{Wilson98}. When
$X\in\{Fi'_{24},B,M\}$, we have checked that $m(S)\geq
|Z||\bN_X(Q)/Q|\geq |\bN_S(P)/P|$, thereby confirming that
Hypothesis~\ref{hypo} holds. For the remaining sporadic groups,
computations using \cite{GAP} reveal that Hypothesis~\ref{hypo}
fails in the following cases. We include here the relevant values of
$m(S)$ and $b(\bN_S(P))$.
\begin{align*}(S,p,b(\bN_S(P)),m(S))\in\{&(Co_2, 5, 24, 23), (Co_3, 3, 32, 23),
(Co_3, 5, 24, 23),\\ &(McL, 5, 24, 22), (3\cdot McL, 5, 24, 22)\}.
\end{align*}
When $S\in\{Co_2,Co_3\}$, as the outer automorphism group of $S$ is
trivial, any bijection (from $\Irr_{p'}(S)$ to
$\Irr_{p'}(\bN_S(P))$) is McKay-good. In these cases, $S$ has a
unique irreducible $p'$-degree character, say $\psi$, of degree
smaller than $b(\bN_S(P))$. Furthermore, we observe that $\bN_S(P)$
is not perfect. (Indeed, $\bN_S(P)/P' \cong 5^2 \rtimes (4 \cdot
\Sy_4)$ when $S = Co_2$, and $\bN_S(P)/P' \cong \Sy_3 \times
(3^2:SD_{16})$ when $S = Co_3$.) Consequently, $\bN_S(P)$ has a
nontrivial linear character, say $\tau$. It then follows that any
bijection sending $\psi$ to $\tau$ satisfies the required degree
condition.

Let $(S,p)\in\{(McL, 5),(3\cdot McL, 5)$\}. In both cases, the outer
automorphism group $\Out(S)\cong C_2$ normalizes $P$. Moreover, $S$
has exactly one irreducible character of degree less than
$b(\bN_S(P))=24$, and this character is necessarily
$\Out(S)$-invariant. Suppose first that $S=McL$. By \cite{Atlas},
$\bN_S(P)=P\rtimes D$, where $P$ is an extra-special group of order
$5^3$ and exponent $5$, and $D=C_3\rtimes C_8$. The group $D$ acts
Frobeniusly on $\Irr(P/P')\cong C_5\times C_5$, and hence $\bN_S(P)$
has a unique irreducible character of degree 24. This is the only
$5'$-degree character of $\bN_S(P)$ of degree greater than
$m(S)=22$. All other $5'$-degree irreducible characters of
$\bN_S(P)$ are trivial on $P$, and thus can be viewed as characters
of $D$, with degrees $1$ and $2$. Now, the required bijection can be
constructed in the way that this degree-24 character corresponds to
(any) $\Out(S)$-invariant irreducible character of $S$ of degree
greater than $22$. (The group $McL$ indeed has several such
characters.) When $S=3\cdot McL$, again $\Irr_{5'}(\bN_S(P))$
contains a single member of degree 24. As in the case $S=McL$, this
is the only $5'$-degree character of $\bN_S(P)$ of degree greater
than $m(S)=22$, and the same reasoning applies.

\medskip

II) The simple groups of Lie type with a non-generic Schur
multiplier are: \begin{align*}&PSL_2(4), PSL_2(9), PSL_3(2),
PSL_3(4), PSL_4(2), PSU_4(2), PSU_4(3),\\ &PSU_6(2),\ta B_2(8),
P\Omega_7(3), PSp_6(2), P\Omega^+_8(2), G_2(3), G_2(4), F_4(2),
\text{ and } \ta E_6(2).\end{align*} (See \cite[Table~24.3]{MT11}.)

Consider the (only) exceptional cover $S=2\cdot F_4(2)$ of
$X=F_4(2)$. When $p=2$, as all the faithful irreducible characters
of $S$ have even degree (\cite{Atlas}), the problem is reduced to
the simple group $F_4(2)$, which was already solved in
Proposition~\ref{prop:defining-char}. When $p=3$, the Sylow
normalizer $\bN_X(Q)$ of $X$ is contained in $PSL_4(3).2_2$ and
lifts to the direct product $C_2\times \bN_X(Q)$ in $S$, as
described in \cite[\S4]{Malle08}. From the proof of
Proposition~\ref{prop:defining-char}, the maximal $3'$-degree of the
Sylow $3$-normalizer of $SL_4(3)$ is at most $(3-1)^3/2=4$, implying
$b(\bN_S(P))=b(\bN_X(Q))\leq 8$. When $p=5$, $\bN_X(Q)$ is an
extension of $C_5^2$ with the complex reflection group $G_8$, which
has the structure $C_4.\Sy_4$. In this case,
$\bN_S(P)=Z\times\bN_X(Q)$, and elementary character theory yields
$b(\bN_S(P))\leq 24$. Similarly, for $p=7$, $\bN_X(Q)$ is an
extension of $C_7^2$ with the complex reflection group $G_5$ of
structure $C_6.\Al_4$. Here, it can be shown that $b(\bN_S(P))\leq
48$. As $m(S)=52$ (\cite{Atlas}), Hypothesis~\ref{hypo} holds for
all the relevant $p$.

Consider $X=\ta E_6(2)$. This group has two exceptional covers
$2\cdot X$ and $6\cdot X$ with cyclic center. %(Since the cover
%$2^2\cdot X$ has no faithful characters and
%Conjecture~\ref{thm:simple} without the degree condition has been
%known for all finite groups, the problem for $2^2\cdot X$ is reduced
%to the one for $2\cdot X$.)
When $p=2$, as in the case $S=2\cdot F_4(2)$, the problem reduces to
handling the covers $X$ and $3\cdot X$, which are actually
non-exceptional covers of $X$, and we are done by
Proposition~\ref{prop:defining-char}. Consider $p=3$. Then
$\bN_X(Q)$ is contained in $\Omega_7(3)$ and is lifted to the direct
product $C_2\times \bN_X(Q)$ in $2\cdot X$. The proof of
Proposition~\ref{prop:defining-char} shows that the maximal
$3'$-degree of the Sylow $3$-normalizer of $\Omega_7(3)$ is at most
$4$, and therefore $b(\bN_S(P))=b(\bN_X(Q))\leq 4$ for $S=2\cdot X$.
For the six-fold cover $S=6\cdot X$, as the $3'$-degree irreducible
characters of both $S$ and $\bN_S(P)$ are trivial on $\bZ(S)_3$, we
still have $b(\bN_S(P))=b(\bN_{S/\bZ(S)_3}(P/\bZ(S)_3))\leq 4$. When
$p=5$, \cite[Table~2]{Malle08} shows that the largest $5'$-degree of
the Sylow normalizer of both $2\cdot X$ and $6\cdot X$ (and of $X$
itself as well) is $48$. %The Sylow $5$-normalizer in $X$ has Schur
%multiplier of order 3 (see the last paragraph of
%\cite[p.~460]{Malle08}), and so this largest degree of $2^2\cdot X$
%is again 48.
For $p=7$, as mentioned in \cite[\S4]{Malle08}, the
Sylow $7$-normalizer of $X$ is contained in $F_4(2)$ and moreover,
$F_4(2)$ is lifted to $2\cdot F_4(2)$ in $2\cdot X$ and to $3\times
2\cdot F_4(2)$ in $6\cdot X$. It follows from the previous paragraph
that $b(\bN_S(P))\leq 48$, whether $S$ is the double or 6-fold cover
of $X$. On the other hand, the smallest nontrivial degree of $6\cdot
\ta E_6(2)$ is $1938$ (\cite{GAP}). Hypothesis~\ref{hypo} again
holds true in this case.

Finally, using \cite{GAP}, we have checked that
Hypothesis~\ref{hypo} also holds for all other exceptional covers of
the simple groups of Lie type listed above, as well as for the
3-fold and 6-fold covers of $\Al_6$ and $\Al_7$. The proof is
complete.
\end{proof}

A few remarks are in order. First,
Propositions~\ref{prop:defining-char}, \ref{prop:cross-char}, and
\ref{prop:1} together imply that Conjecture~\ref{thm:simple} holds
for every quasisimple group \(S\) of exceptional Lie type and for
every prime \(p\). Second, our arguments show that, in all cases
considered, Hypothesis~\ref{hypo} holds for \((S,p)\), except
possibly in the following situations: \(S=PSU_n(q)\) with \(n\ge 3\)
odd, \((n,q+1)=1\), and \(q\) a power of \(p\); when $S$ is sporadic
and $|S|_p=p$; or when $(S,p)\in\{(Co_2,5), (Co_3,3), (Co_3,5),
(McL,5), (3\cdot McL,5)\}$.

%%%%%%%%%%%%%%%%%%%%%%%%%%%%%%%%%%%%%%%%%%%%%%%%%%%%%%

\section{Odd-degree characters}\label{sec:p=2}

In this section we confirm Conjectures~\ref{thm:simple},
\ref{conj:inductive Giannelli-McKay}, and \ref{conj:Giannelli-McKay}
for $p=2$, thereby proving Theorem~B.

\begin{thm}\label{thm:p=2}
Conjecture~\ref{thm:simple} is true for all quasisimple groups $S$
and $p=2$.
\end{thm}

\begin{proof}
By Propositions \ref{prop:1}, \ref{prop:defining-char}, and
\ref{prop:cross-char}, and noting that Sylow $2$-subgroups of an
alternating group or its double cover are self-normalizing
\cite{Olsson76}, we only need to consider classical groups over
fields of odd characteristic. Furthermore, we may assume that $S$ is
a non-exceptional covering group of the simple group $X:=S/\bZ(S)$.
As before, we use $G$ for the finite reductive group of
simply-connected type such that $G/\bZ(G)=X$, and so $S$ is a
certain quotient of $G$. By Proposition~\ref{pro:key}, the existence
of a required bijection is guaranteed if we are able to show that
\begin{equation}\label{eq:2}b(\bN_G(Q))\leq m(G)\end{equation} for $Q\in\Syl_2(G)$.

Let $X=PSL_2(q)$ with $5\leq q$ odd, and hence $G=SL_2(q)$. Assume
first that $q\equiv \pm 3(\bmod\,8)$. Then $Q$ is the quaternion
group of order $8$ and $\bN_G(Q)$ is isomorphic to $SL_2(3)$ (see
\cite[\S15E]{IMN07}). We have $b(\bN_G(Q))=3$, while $m(G)=(q-1)/2$
if $q\equiv 3(\bmod\,8)$ and $m(G)=(q+1)/2$ if $q\equiv
-3(\bmod\,8)$, and so (\ref{eq:2}) is satisfied. When $q\equiv \pm
1(\bmod\,8)$, $Q$ is self-normalizing and the inequality is trivial.

Let $X=PSL^\pm_n(q)$ with $n\geq3$. Here, as usual, we use the
superscript $+$ for linear groups, while $-$ for unitary groups.
Then $G=SL^\pm_n(q)$. Let $\widetilde{G}:=GL_n^\pm(q)$,
$R\in\Syl_2(\widetilde{G})$, and take $Q:=R\cap G$. By
\cite[Theorem~1]{Kon05}, we have
\[\bN_{\widetilde{G}}(Q)=\bN_{\widetilde{G}}(R)=R\bC_{\widetilde{G}}(R).\]
The structure of Sylow normalizers in $\widetilde{G}$ was determined
by Carter and Fong in \cite[Lemma~6 and Theorem~4]{CF64}, as follows
\[
\bN_{\widetilde{G}}(R)\cong R\times (C_{(q \mp 1)_{2'}})^t,
\]
where $(q\mp 1)_{2'}$ is the odd part of $q\mp 1$ and $t$ is the
number of terms in the $2$-adic expansion of $n$. It follows that
\[\bN_{\widetilde{G}}(Q)/Q'=\bN_{\widetilde{G}}(R)/Q'\cong (R/Q')\times (C_{(q\mp
1)_{2'}})^t.\] Let $\textbf{b}(M)$ denote the largest degree of an
irreducible character of a finite group $M$. We have
\[
\textbf{b}(\bN_{\widetilde{G}}(Q)/Q')= \textbf{b}(R/Q')\leq
|R:Q|\leq q+1.
\]
Here, the inequality in the middle follows from
\cite[Corollary~11.29]{Is} and the fact that $Q/Q'$ is an abelian
normal subgroup of $R/Q'$. The last inequality follows from
$|R:Q|=|R:(R\cap G)|=|RG:G|\leq |\widetilde{G}:G|\leq q+1$. Since
$\bN_{G}(Q)/Q'$ is a normal subgroup of $\bN_{\widetilde{G}}(Q)/Q'$,
we deduce that
\[b(\bN_{G}(Q))= \textbf{b}(\bN_{G}(Q)/Q')\leq
\textbf{b}(\bN_{\widetilde{G}}(Q)/Q')\leq q+1.\] The desired
inequality $b(\bN_G(Q))\leq m(G)$ then follows immediately from the
bound provided in Table~\ref{table}.

Next, consider $X=PSp_{2n}(q)$ with $n\geq 2$ and $q$ odd. Then
$G=Sp_{2n}(q)$ with $\bZ(G)$ being cyclic of order $2$. The Sylow
$2$-subgroup $Q$ is self-normalizing in $G$ when $q\equiv \pm
1(\bmod\,8)$; otherwise, $|\bN_G(Q)/Q|=3^t$ where $t$ is the number
of terms in the $2$-adic expansion of $n$ (see
\cite[Theorem~4]{CF64}). In the former case, Hypothesis~\ref{hypo}
is trivial. In the latter, we have
\[
b(\bN_G(Q))\leq |\bN_G(Q)/Q| =3^t<(q^n-1)/2=d(G)\leq m(G),
\]
and we are done again.

Lastly, consider $X=\Omega_{2n+1}(q)$ with $n\geq 3$ or
$X=P\Omega_{2n}^\pm(q)$ with $n\geq 4$ ($q$ again is odd). Here
$G=Spin_{2n+1}(q)$ or $Spin_{2n}^\pm(q)$, respectively. According to
\cite[Theorem~5]{CF64}, Sylow $2$-subgroups of $G$ are
self-normalizing, and we conclude as before.
\end{proof}

\begin{thm}\label{thm:p=2Giannelli}
Conjecture~\ref{conj:Giannelli-McKay} is true when $p=2$.
\end{thm}

\begin{proof}
This follows from Theorems \ref{thm:p=2} and \ref{thm:GMC
reduction}.
\end{proof}

As mentioned in the introduction, Theorem~B immediately follows from
Theorem~\ref{thm:p=2Giannelli}.

We conclude with some remarks. To complete the proof of
Conjecture~\ref{conj:A} and Conjecture~\ref{conj:Giannelli-McKay}
for all primes $p$, it remains to verify
Conjecture~\ref{conj:inductive Giannelli-McKay} for quasisimple
classical groups $S$ in characteristic different from $p$, as well
as for covers of alternating groups. This appears to be a nontrivial
problem, since for these groups the normalizer $\bN_S(P)$ typically
has many irreducible $p'$-characters whose degrees exceed the
minimal $p'$-degree of $S$.

Conjecture~\ref{conj:Giannelli-McKay} has now been established for
symmetric groups by Giannelli \cite{G}. It may be possible to adapt
the methods of \cite{G} to prove Conjecture~\ref{conj:inductive
Giannelli-McKay} for alternating groups.

%%%%%%%%%%%%%%%%%%%%%%%%%%%%%%%%%%%%%

%%%%%%%%%%%%%%%%%%%%%%%%%%%%%%%%%%%%%%%
%%%%%%%%%%%%%%%%%%%%%%%%%%%%%%%%%%%%%%%%%
%%%%%%%%%%%%%%%%%%%%%%%%%%%%%%%%%%%%%%%%%%%


\begin{thebibliography}{ABCDEF}

\bibitem[A94]{An94}
J. An, Alperin-McKay conjecture for the Chevalley groups $G_2(q)$,
\emph{J. Algebra} \textbf{165} (1994), no. 1, 184--193.

\bibitem[A98]{An98}
J. An, The Alperin and Dade conjectures for Ree groups $\ta
F_4(q^2)$ in non-defining characteristics, \emph{J. Algebra}
\textbf{203} (1998), no. 1, 30--49.

\bibitem[B63]{B}
R. Brauer, Representations of finite groups, in {\em Lectures on
modern mathematics}, Wiley \& Sons (1963), 133--175.

\bibitem[BM92]{BM92}
M. Brou\'{e} and G. Malle, Th\'{e}or\`{e}mes de Sylow
g\'{e}n\'{e}riques pour les groupes r\'{e}ductifs sur les corps
finis, \emph{Math. Ann.} \textbf{292} (1992), 241--262.

\bibitem[BM93]{BM93}
M. Brou\'{e} and G. Malle, Zyklotomische
Heckealgebren,
Repr\'{e}sentations unipotentes g\'{e}n\'{e}riques et blocs des
groupes r\'{e}ductifs finis, \emph{Ast\'{e}risque} \textbf{212}
(1993), 119--189.

\bibitem[B09]{Brunat09}
O. Brunat, On the inductive McKay condition in the defining
characteristic, \emph{Math. Z.} \textbf{263} (2009), no. 2,
411--424.


\bibitem[CS24]{CS}
M. Cabanes and B. Sp\"ath, The McKay conjecture on character
degrees, https://arxiv.org/abs/2410.20392, to appear in
 {\em Ann. of Math.}

\bibitem[C81]{Carter81}
R.\,W. Carter, Centralizers of semisimple elements in the finite
classical groups, \emph{Proc. London Math. Soc.} \textbf{42} (1981),
1--41.

\bibitem[C85]{Carter85}
R.\,W. Carter, \emph{Finite groups of Lie type. Conjugacy classes
and complex characters}, Pure Appl. Math. (N. Y.) Wiley-Intersci.
Publ. John Wiley \& Sons, Inc., New York, 1985.

\bibitem[CF64]{CF64}
R. Carter and P. Fong, The Sylow $2$-subgroups of the finite
classical groups, \emph{J. Algebra} \textbf{1} (1964), 139--151.



\bibitem[Atl]{Atlas}
J.\,H. Conway, R.\,T. Curtis, S.\,P. Norton, R.\,A. Parker, and
R.\,A. Wilson, \emph{Atlas of finite groups}, Oxford University
Press, London, 1985.

\bibitem[D66]{Dade66}
{\sc E.\,C. Dade}, Blocks with cyclic defect groups, \emph{Ann. of
Math.} \textbf{84} (1966), 20--48.

\bibitem[D96]{Dade96}
{\sc E.\,C. Dade}, Counting characters in blocks with cyclic defect
groups. I, \emph{J. Algebra} \textbf{186} (1996), 934--969.

\bibitem[D18]{Fra20}
G. De Franceschi, \emph{Centralizers and conjugacy classes in finite
classical groups}, DPhil Thesis, The University of Auckland, 2018.


\bibitem[F80]{F}
W. Feit, Some consequences of the classification of finite simple
groups, The Santa Cruz conference on finite groups, {\em Proc.
Sympos. Pure Math.} {\bf 37}, Amer. Math. Soc., Providence, RI (1980),
 175--181.

\bibitem[F82]{Feit82}
W. Feit, \emph{The representation theory of finite groups},
North-Holland, Amsterdamr-New York-Oxford, 1982.

\bibitem[GAP]{GAP}
{\sc The GAP Group}, \emph{GAP Groups, Algorithms, and Programming,
Version 4.11.0}, 2020. (\url{http://www.gap-system.org})

\bibitem[GM20]{GM20}
M. Geck and G. Malle, \emph{The character theory of finite groups of
Lie type}, Cambridge Stud. Adv. Math. \textbf{187}, Cambridge
University Press, Cambridge, 2020.

 \bibitem[G25]{G}
E. Giannelli, McKay bijections and character degrees, preprint,
2025. Preprint at https://arxiv.org/abs/2507.01730.

 \bibitem[I86]{IsCR}
 I.\,M. Isaacs, Subgroups with the character restriction property,
{\em J. Algebra} {\bf 100} (1986), 403--420.

\bibitem[I06]{Is}
{I.\,M. Isaacs}, {\em Character theory of finite groups}, AMS
Chelsea Publishing, Providence, RI, 2006.

\bibitem[I08]{Is08}
{I.\,M. Isaacs}, {\em Finite Group Theory}, AMS,
 Providence, RI, 2008.

\bibitem[IMN07]{IMN07}
I.\,M. Isaacs, G. Malle, and G. Navarro, A reduction theorem for the
McKay conjecture, \emph{Invent. Math.} \textbf{170} (2007), no. 1,
33--101.

\bibitem[K05]{Kon05}
A.\,S. Kondrat'ev, Normalizers of Sylow $2$-subgroups in finite
simple groups, \emph{Math. Notes} \textbf{78} (2005), no. 3-4,
338--346.

\bibitem[KS16]{KS16}
S. Koshitani and B. Sp\"{a}th, The inductive Alperin-McKay and
blockwise Alperin weight conditions for blocks with cyclic defect
groups and odd primes, \emph{J. Group Theory} \textbf{19} (2016),
no. 5, 777--813.

\bibitem[LOST10]{LOST10}
M.\,W. Liebeck, E.\,A. O'Brien, A. Shalev, and P.\,H. Tiep, The Ore
conjecture, \emph{J. Eur. Math. Soc.} \textbf{12} (2010), no. 4,
939--1008.

\bibitem[Li25]{Li25}
Markus Linckelmann, On character degrees of cyclic and Klein four
defect blocks, preprint, 2025. Preprint at
https://arxiv.org/abs/2509.01394.

\bibitem[L\"{u}01]{Lubeck01}
F. L\"{u}beck, Smallest degrees of representations of exceptional
groups of Lie type, \emph{Comm. Algebra} \textbf{29} (2001), no. 5,
2147--2169.

\bibitem[MT11]{MT11}
G. Malle and D. Testerman, \emph{Linear algebraic groups and finite
groups of Lie type}, Cambridge Stud. Adv. Math. \textbf{133},
Cambridge University Press, Cambridge, 2011.

\bibitem[M07]{Malle07}
G. Malle, Height $0$ characters of finite groups of Lie type,
\emph{Represent. Theory} \textbf{11} (2007), 192--220.

\bibitem[M08]{Malle08}
G. Malle, The inductive McKay condition for simple groups not of Lie
type, \emph{Comm. Algebra} \textbf{36} (2008), no. 2, 455--463.

%\bibitem[MW93]{MW93}
%O. Manz and T.~R. Wolf, {\it Representations of solvable groups}, London Mathematical Society Lecture Note Series, 185, Cambridge Univ. Press, Cambridge, 1993.

\bibitem[N04a]{N03}
G. Navarro, Problems on characters and Sylow subgroups. {\em Finite
groups 2003}, 275--281, Walter de Gruyter, Berlin, 2004.

%\bibitem[N04b]{Navarro04}
%{\sc G. Navarro}, The McKay conjecture and Galois automorphisms,
%\emph{Ann. of Math.} \textbf{160} (2004), 1129--1140.

\bibitem[N18]{N}
G. Navarro, {\it Character theory and the McKay conjecture},
Cambridge University Press, Cambridge, 2018.

\bibitem[N10]{Nguyen10}
H.\,N. Nguyen, Low-dimensional complex characters of the symplectic
and orthogonal groups, \emph{Comm. Algebra}~\textbf{38} (2010),
1157--1197.

\bibitem[O76]{Olsson76}
J. Olsson, McKay numbers and heights of characters, \emph{Math.
Scand.}, \textbf{38} (1976), 25--42.

\bibitem[S12]{Spath12}
B. Sp\"{a}th, Inductive McKay condition in defining characteristic,
\emph{Bull. Lond. Math. Soc.} \textbf{44} (2012), no. 3, 426--438.

\bibitem[T18]{Taylor18}
J. Taylor, Action of automorphisms on irreducible characters of
symplectic groups, \emph{J. Algebra} \textbf{505} (2018), 211--246.

\bibitem[TZ96]{TZ96}
P.\,H. Tiep and A.\,E. Zalesski, Minimal characters of finite
classical groups, \emph{Comm. Algebra} \textbf{24} (1996),
2093--2167.

\bibitem[T08]{T08}
A. Turull, Above the Glauberman correspondence, {\em Adv. Math.} {\bf 217} (2008), no.~5, 2170--2205.

\bibitem[T09]{T09}
A. Turull, The Brauer-Clifford group, {\em J. Algebra} {\bf 321}
(2009), no.~12, 3620--3642.



\bibitem[W98]{Wilson98}
R.\,A. Wilson, The McKay conjecture is true for the sporadic simple
groups, \emph{J. Algebra} \textbf{207} (1998), 294--305.

 \end{thebibliography}
\end{document}